\documentclass[10pt]{article}
\usepackage{graphicx}
\usepackage{epsfig}
\usepackage{amssymb}
\usepackage{amsmath}
\usepackage{amsthm}
\newtheorem{lemma}{Lemma}[section]
\newtheorem{theorem}{Theorem}[section]

{\end{description}}
{\end{description}}

{\hfill $\bullet$ \\}

\newcommand{\be}{\begin{equation}}
\newcommand{\ee}{\end{equation}}

\begin{document}
    \title{A Time-Split MacCormack Scheme for Two-Dimensional Nonlinear Reaction-Diffusion Equations}
   \author{\Large{Eric Ngondiep}
       \thanks{Tel.: +966506048689. E-mail addresses:\ ericngondiep@gmail.com or engondiep@imamu.edu.sa
        (Eric Ngondiep).\ }}
   \date{\small{Department of Mathematics and Statistics, College of Science, Al-Imam Muhammad
   Ibn Saud Islamic University (IMSIU), 90950 Riyadh, Kingdom of Saudi Arabia}\\
       \text{\,}\\
       \small{Hydrological Research Centre, Institute for Geological and Mining Research, 4110 Yaounde-Cameroon}}

    \maketitle
   \textbf{Abstract.}
   A three-level explicit time-split MacCormack scheme is proposed for solving the two-dimensional nonlinear reaction-diffusion equations. The computational cost is reduced thank to the splitting and the explicit MacCormack scheme. Under the well known condition of Courant-Friedrich-Lewy (CFL) for stability of explicit numerical schemes applied to linear parabolic partial differential equations, we prove the stability and convergence of the method in $L^{\infty}(0,T;L^{2})$-norm. A wide set of numerical evidences which provide the convergence rate of the new algorithm are presented and critically discussed.
   \text{\,} \\
    \text{\,}
   \\
   {\bf AMS Subject Classification (MSC). 65M10, 65M05}.

      \section{Introduction and motivation}\label{sec1}
      A large number of biological problems of significant interest are modeled by parabolic equations \cite{27wcld}. The general framework is a set
      of biological entities (either ions, molecules, proteins or cells) that interact with each other and diffuse within a given domain. So it becomes
      possible to build some models via reaction-diffusion equations. For example, the dendritic spines possess a twitching motion which are described
       by the reaction-diffusion models \cite{34}. In this paper, we consider the following two-dimensional reaction-diffusion equations,
      \begin{equation}\label{1}
        u_{t}-a\Delta u=f(u),\text{\,\,\,\,\,}(x,y)\in\Omega,\text{\,\,\,\,\,}t\in(0,T];
      \end{equation}
      with the initial condition
      \begin{equation}\label{2}
        u(x,y,0)=u_{0}(x,y),\text{\,\,\,\,\,}(x,y)\in\overline{\Omega};
      \end{equation}
      and the boundary condition
      \begin{equation}\label{3}
        u(x,y,t)=\varphi(x,y,t),\text{\,\,\,\,\,}(x,y)\in\partial\Omega,
        \text{\,\,\,\,\,}t\in(0,T];
      \end{equation}
      where $a$ is the diffusive coefficient, $f\in\mathcal{C}^{1}(\mathbb{R})$ is a Lipschitz function, $\Omega=(0,1)^{2},$ $\Delta$ denotes the Laplacian
      operator, $\partial\Omega$ is the boundary of $\Omega$ and $u_{t}$ designates $\frac{\partial u}{\partial t}$. The initial condition $u_{0}$
      and the boundary condition $\varphi$ are assumed to be regular enough and satisfy the requirement $\varphi(x,y,0)=u_{0}(x,y),$ for every
      $(x,y)\in\partial\Omega,$ so that the initial value problem $(\ref{1})$-$(\ref{3}),$ admits a smooth solution.\\

      In the last decades \cite{1rb,mc2,nrn}, MacCormack approach which is a predictor-corrector, finite difference scheme has been used to solve
      certain classes of nonlinear partial differential equations (PDEs). There exist both explicit and implicit versions of the method, but the
      explicit predates the implicit by more than a decade, and it is considered as one of the milestones of computational fluid dynamics. Both versions facilitate the solution of parabolic and hyperbolic equations by marching forward in time \cite{mc2,mc3,mc4}. The popularity of MacCormack explicit
      method is due in part to its simplicity and ease of implementation. The predictor and corrector phases each uses forward differencing for first-order
       time derivatives, with alternate one-side differencing for first-order space derivatives. This is especially convenient for systems of equations with      nonlinear advertive jacobian matrices associated with one-side explicit schemes, such as Lax-Wendroff approach  (for instance, see \cite{3rb,nnnn,nrn}). However, the explicit MacCormack is not a suitable method for solving high Reynolds numbers flows, where the viscous regions become very thin (see \cite{apt}, P. $630$). To overcome this difficulty, MacCormack \cite{mc1} developed a hybrid version of his scheme, known as the MacCormack rapid solver method. The new algorithm is an explicit-implicit method. For example, in a search of an efficient solution, the authors \cite{en1,en2,en3,en4} applied this hybrid method to some complex PDEs (such as: mixed Stokes-Darcy model and $2D$ incompressible Navier-Stokes equations) and they obtained satisfactory results regarding both stability and convergence rate of the method. It is worth noticing to mention that the rapid solver algorithm has a good stability condition and it is too faster than a large set of numerical methods for solving steady and unsteady flows at high to low Reynolds numbers \cite{mc1}. So, the hybrid method of MacCormack will be used to solve the $2$D reaction-diffusion equations $(\ref{1})$-$(\ref{3}),$ in our future works.\\

      Armed with the information gleaned from both MacCormack and MacCormack rapid solver methods, we can now analyze a time-split MacCormack technique applied to problem $(\ref{1})$-$(\ref{3}).$ Firstly, it's worth noting to recall that the problem considered in this paper has been solved in literature by a wide set of explicit, implicit and coupled explicit-implicit numerical schemes. While some explicit methods usually suffer the severely restricted temporal step size \cite{1wcld,2wcld}, the fully implicit methods although unconditionally stable, provide a large system of nonlinear equations at every time level \cite{3wcld,4wcld}. These systems lead to a considerable computational cost in practical applications. A possible improvement is to use the second time discretization such as the linearized Cank-Nicolson, implicit-explicit and collocation approaches \cite{7wcld,8wcld,14wcld,en,17wcld,18wcld,24wcld,27wcld,28wcld}. Unfortunately, at least two starting values are needed to begin these algorithms. These values can be obtained by the initial and boundary conditions and an additional iterative method. To overcome this difficulty, the authors \cite{wcld} applied a two-level linearized compact ADI scheme to problem $(\ref{1})$-$(\ref{3}).$ The main results of their work (namely Theorems $1$-$2$) have been proved under the assumptions that the time step $\Delta t=k,$ mesh size $\Delta x=\Delta y=h$ and the ratio $\frac{\Delta t}{h^{2}}\leq C,$ must be sufficiently small (\cite{wcld}, page $9$, line above $(3.26)$ and page $10,$ Theorem $2$). These requirements in general are less restrictive and can make the method even more impractical. Furthermore, this paper represents an extension of the work in \cite{en5}.\\

      The time-split MacCormack approach we study for the initial-boundary value problem $(\ref{1})$-$(\ref{3})$ is new, a three-level explicit predictor-corrector method, second order accurate in time and fourth order convergent in space, under the time step restriction: $\frac{ak}{h^{2}}\leq\frac{1}{2},$ and it is motivated by this time step restriction (indeed, lots of explicit schemes for solving equation $(\ref{1})$-$(\ref{3}),$ are stable under the well-known condition of Courant-Friedrich-Lewy: $\frac{4ak}{h^{2}}\leq1$) and its efficiency and effectiveness. From this observation, it is obvious that: (a) a time-split MacCormack approach is more practical, (b) although the new algorithm and a two-level linearized compact ADI method have the same order of convergence, the linearized compact ADI scheme requires substantially more computer times to solve problem $(\ref{1})$-$(\ref{3}),$ than does a time-split MacCormack. An explicit time-split MacCormack algorithm \cite{mc4,en5,mc5} "splits" the original MacCormack scheme into a sequence of one-dimensional operations, thereby achieving a good stability condition. In other words, the splitting makes it possible to advance the solution in each direction with the maximum allowable time step. This is particularly advantageous if the allowable time steps $\Delta t_{x}$ and $\Delta t_{y}$, are much different because of differences in the mesh spacings $\Delta x$ and $\Delta y.$ In order to explain this method, we will make use of the $1$D difference operators $L_{x}(\Delta t_{x})$ and $L_{y}(\Delta t_{y})$. Setting $u_{ij}^{n}=u(x_{i},y_{j},t^{n}),$ the $L_{x}(\Delta t_{x})$ operator applied to $u_{ij}^{n}$,
      \begin{equation}\label{8}
        u_{ij}^{*}=L_{x}(\Delta t_{x})u_{ij}^{n},
       \end{equation}
      is by definition equivalent to the two-step predictor-corrector MacCormack formulation. The $L_{y}(\Delta t_{y})$ operator is defined in a similar manner, that is,
      \begin{equation}\label{9}
        u_{ij}^{*}=L_{y}(\Delta t_{y})u_{ij}^{n}.
       \end{equation}
       These expressions make use of a dummy time index, which is denoted by the asterisk. Now, letting $\Delta t_{x}=\Delta t$ and
        $\Delta t_{y}=\frac{\Delta t}{2m},$ where $m$ is a positive integer, a second order accurate scheme can be constructed by applying the
        $L_{x}$ and $Ly$ operators to $u_{ij}^{n},$ in the following way:
       \begin{equation*}
        u_{ij}^{n+1}=\left[L_{y}\left(\frac{\Delta t}{2m}\right)\right]^{m}L_{x}(\Delta t)\left[L_{y}\left(\frac{\Delta t}{2m}\right)\right]^{m}u_{ij}^{n}.
      \end{equation*}
       This sequence is quite useful for the case $\Delta y<<\Delta x.$ In general, a scheme formed by a sequence of these operators is: $(1)$ stable,
       if the time step of each operator does not exceed the allowable time step for that operator; $(2)$ consistent, if the sums of the time steps for
       each of the operators are equal: and $(3)$ second-order accurate, if the sequence is symmetric.\\

       In this paper, we are interested in a numerical solution of the initial-boundary value problem $(\ref{1})$-$(\ref{3}),$
       using a time-split MacCormack approach. Specifically, the work is focused on the following four items:
       \begin{description}
         \item[1.] full description of a three-level explicit time-split MacCormack scheme for solving the nonlinear reaction-diffusion
         equations $(\ref{1})$-$(\ref{3});$
         \item[2.] stability analysis of the numerical scheme;
         \item[3.] error estimates of the method;
         \item[4.] a wide set of numerical examples which provide the convergence rate, confirms the theoretical results and shows the
         efficiency and effectiveness of the method.
       \end{description}
        Items $1$, $2$ and $3$, are our original contributions since as far as we know, there is no available work in literature which solves the reaction-diffusion model $(\ref{1})$-$(\ref{3}),$ using a time-split MacCormack method.\\

        The paper is organized as follows: Section $\ref{sec2}$ considers a detailed description of a three-level explicit time-split MacCormack method applied to problem $(\ref{1})$-$(\ref{3}).$ In section $\ref{sec3},$ we study the stability of the numerical scheme under the condition given above, while section $\ref{sec4}$ analyzes the error estimates and the convergence of the method. A large set of numerical examples which provides the convergence rate of the new algorithm and confirms the theoretical result (on the stability) are presented and discussed in section $\ref{sec5}.$ We draw the general conclusion and present our future works in section $\ref{sec6}.$

      \section{Full description of a time-split MacCormack method}\label{sec2}
       This section deals with the description of a three-level explicit time-split MacCormack method applied to two-dimensional nonlinear reaction-diffusion equations $(\ref{1})$-$(\ref{3}).$\\

       Let $N$ and $M$ be two positive integers. Set $k:=\Delta t=\frac{T}{N};$ $h:=\Delta x=\Delta y=\frac{1}{M},$ be the time step and mesh size, respectively. Put $t^{n}=kn,$ $t^{*}=(n+r)k,$ $t^{**}=(n+s)k,$ where $0<r<s<1,$ so $t^{*}\in(t^{n},t^{n+1}),$ $t^{**}\in(t^{*},t^{n+1});$ $n=0,1,2,...,N-1;$ $x_{i}=ih;$ $y_{j}=jh;$ $0\leq i,j\leq M$. Also, let $\Omega_{k}=\{t^{n},0\leq n\leq N\};$ $\overline{\Omega}_{h}=\{(x_{i},y_{j}),0\leq i,j\leq M\};$
       $\Omega_{h}=\overline{\Omega}_{h}\cap\Omega$ and $\partial\Omega_{h}=\overline{\Omega}_{h}\cap\partial\Omega.$\\

       Consider $\mathcal{U}_{h}=\{u_{ij}^{n},n=0,1,...,N;\text{\,}i,j=0,1,2,...,M\}$ be the space of grid functions defined on $\Omega_{h}\times\Omega_{k}.$ Letting
       \begin{equation*}
        \delta_{t}u_{ij}^{*}=\frac{u_{ij}^{*}-u_{ij}^{n}}{k/2};\text{\,\,\,} \delta_{t} u_{ij}^{**}=\frac{u_{ij}^{**}-u_{ij}^{*}}{k};
        \text{\,\,\,}\delta_{t} u_{ij}^{n+1}=\frac{u_{ij}^{n+1}-u_{ij}^{**}}{k/2};\text{\,\,\,}\delta_{x}u_{i+\frac{1}{2},j}^{n}=\frac{u_{i+1,j}^{n}
        -u_{ij}^{n}}{h};
       \end{equation*}
       \begin{equation}\label{4}
        \delta_{y}u_{i,j+\frac{1}{2}}^{n}=\frac{u_{i,j+1}^{n}-u_{ij}^{n}}{h};\text{\,\,\,}\delta_{x}^{2}u_{ij}^{n}=\frac{\delta_{x}u_{i+\frac{1}{2},j}^{n}
        -\delta_{x}u_{i-\frac{1}{2},j}^{n}}{h};\text{\,\,}\delta_{y}^{2}u_{ij}^{n}=\frac{\delta_{y}u_{i,j+\frac{1}{2}}^{n}-\delta_{y}u_{i,j-\frac{1}{2}}^{n}}{h}.
       \end{equation}
       Using this, we define the following norms and scalar products.
       \begin{equation*}
        \|u^{n}\|_{L^{2}}=h\left(\underset{i,j=1}{\overset{M-1}\sum}|u_{ij}^{n}|^{2}\right)^{\frac{1}{2}};\text{\,\,}\|\delta_{x}u^{n}\|_{L^{2}}=
        h\left(\underset{j=1}{\overset{M-1}\sum}\underset{i=0}{\overset{M-1}\sum}|\delta_{x}u_{i+\frac{1}{2},j}^{n}|^{2}\right)^{\frac{1}{2}};
       \end{equation*}
       \begin{equation}\label{5}
        \|\delta_{y}u^{n}\|_{L^{2}}=h\left(\underset{j=0}{\overset{M-1}\sum}\underset{i=1}{\overset{M-1}\sum}|\delta_{y}u_{i,j+\frac{1}{2}}^{n}|^{2}\right)^{\frac{1}{2}};
        \text{\,}\|\delta_{\lambda}^{2}u^{n}\|_{L^{2}}=h\left(\underset{i,j=1}{\overset{M-1}\sum}|\delta_{\lambda}^{2}u_{ij}^{n}|^{2}\right)^{\frac{1}{2}},
       \end{equation}
       where $\lambda=x$ or $y.$ Furthermore, the scalar products are defined as
       \begin{equation*}
        (u^{n},v^{n})=h^{2}\underset{i,j=1}{\overset{M-1}\sum}u^{n}_{ij}v_{ij}^{n};  \text{\,\,}<\delta_{x}u^{n},\delta_{x}v^{n}>_{x}=h^{2}\underset{j=1}
        {\overset{M-1}\sum}\underset{i=0}{\overset{M-1}\sum}\delta_{x}u_{i+\frac{1}{2},j}^{n}\delta_{x}v_{i+\frac{1}{2},j}^{n};
       \end{equation*}
       and
       \begin{equation}\label{6}
        <\delta_{y}u^{n},\delta_{y}v^{n}>_{y}=h^{2}\underset{j=0}{\overset{M-1}\sum}\underset{i=1}{\overset{M-1}\sum}\delta_{y}u_{i,j
        +\frac{1}{2}}^{n}\delta_{y}v_{i,j+\frac{1}{2}}^{n}.
       \end{equation}
       The space $H^{1}(\Omega)$ is endowed with the norm $|\cdot|_{H^{1}}$ (respectively, $\|\cdot\|_{H^{1}}$) defined as
       \begin{equation}\label{7}
        |u^{n}|_{H^{1}}=\left(\|\delta_{x}u^{n}\|_{L^{2}}^{2}+\|\delta_{y}u^{n}\|_{L^{2}}^{2}\right)^{\frac{1}{2}}
        \text{\,\,\,\,and\,\,\,\,}\|u^{n}\|_{H^{1}}=\left(\|u^{n}\|_{L^{2}}^{2}+\|\delta_{x}u^{n}\|_{L^{2}}^{2}
        +\|\delta_{y}u^{n}\|_{L^{2}}^{2}\right)^{\frac{1}{2}}.
       \end{equation}
       It is worth noticing to mention that a time-split MacCormack \cite{mc5,mc4} "splits" the original MacCormack scheme into a sequence of $1$D operators, thereby achieving a less restrictive stability condition. In order words, the splitting makes it possible to advance the solution in each direction with the maximum allowable time step (\cite{apt}, page $231$).\\

       In other to give a detailed description of this method, we consider the $1$D difference operators $L_{x}(\Delta t_{x})$ and $L_{y}(\Delta t_{y})$
       defined by equations $(\ref{8})$ and $(\ref{9})$, respectively. Following the approach presented in (\cite{apt}, page $231$), a second-order accurate       scheme can be constructed by applying the $L_{x}$ and $L_{y}$ operators to $u_{ij}^{n}$ in the following manner:
       \begin{equation}\label{10}
        u_{ij}^{n+1}=L_{y}(k/2)L_{x}(k)L_{y}(k/2)u_{ij}^{n}.
       \end{equation}
       Using these tools, we are able to provide a three-level explicit time-split MacCormack method for solving the initial-boundary value problem $(\ref{1})$-$(\ref{3})$. Putting $\Delta t_{x}=k,$ $\Delta t_{y}=\frac{k}{2}$ and $\Delta x=\Delta y:=h,$ it comes from equations $(\ref{8})$, $(\ref{9})$ and $(\ref{10})$ that
       \begin{equation}\label{11}
       u_{ij}^{*}=L_{y}(k/2)u_{ij}^{n};\text{\,\,}u_{ij}^{**}=L_{x}(k)u_{ij}^{*}=L_{x}(k)L_{y}(k/2)u_{ij}^{n}\text{\,\,\,\,\,and\,\,\,\,\,} u_{ij}^{n+1}=L_{y}(k/2)u_{ij}^{**}.
       \end{equation}
       In the following, we should find explicit expressions of equations $u_{ij}^{*}=L_{y}(k/2)u_{ij}^{n}$ and $u_{ij}^{**}=L_{x}(k)u_{ij}^{*}.$ This will help to give an explicit formula of the equation $u_{ij}^{n+1}=L_{y}(k/2)u_{ij}^{**},$ which represents a "one-step time-split MacCormack algorithm".       For the sake of simplicity, we use both notations: $u)_{ij}^{n}=u_{ij}^{n}$ and $[u+v]_{ij}^{n}=u_{ij}^{n}+v^{n}_{ij}$.\\

       The application of the Taylor series expansion about $(x_{i},y_{j},t^{n})$ at the predictor and corrector
       steps with time step $k/2$ yields
       \begin{equation}\label{12a}
        u_{ij}^{\overline{*}}=u^{n}_{ij}+\frac{k}{2}u_{t})_{ij}^{n}+\frac{k^{2}}{8}u_{2t})_{ij}^{n}+O(k^{3});\text{\,\,}u_{ij}^{\overline{\overline{*}}}=
        u^{n}_{ij}+\frac{k}{2}u_{t})_{ij}^{\overline{*}}+\frac{k^{2}}{8}u_{2t})_{ij}^{\overline{*}}+O(k^{3}).
       \end{equation}
       From the definition of the operator $L_{y}(k/2),$ let consider the equation
       \begin{equation}\label{11a}
        u_{t}-au_{yy}=f(u),\text{\,\,\,which\,\,\,is\,\,\,equivalent\,\,\,to\,\,\,}u_{t}=au_{yy}+f(u).
       \end{equation}
        Using equation $(\ref{11a}),$ it is not difficult to see that
        \begin{equation*}
         u_{2t}=\left(au_{yy}+f(u)\right)_{t}= a^{2}u_{4y}+a\left(f(u)\right)_{yy}+\left(au_{yy}+f(u)\right)f^{'}(u).
        \end{equation*}
         This fact together with equation $(\ref{12a})$ provide
        \begin{equation}\label{12}
        u_{ij}^{\overline{*}}=u^{n}_{ij}+\frac{k}{2}[au_{yy}+f(u)]_{ij}^{n}+\frac{k^{2}}{8}\left[a^{2}u_{4y}+a\left(f(u)\right)_{yy}
        +\left(au_{yy}+f(u)\right)f^{'}(u)\right]_{ij}^{n}+O(k^{3});
       \end{equation}
       \begin{equation}\label{13}
        u_{ij}^{\overline{\overline{*}}}=u^{n}_{ij}+\frac{k}{2}[au_{yy}+f(u)]_{ij}^{\overline{*}}+\frac{k^{2}}{8}\left[a^{2}u_{4y}
        +a\left(f(u)\right)_{yy}+\left(au_{yy}+f(u)\right)f^{'}(u)\right]_{ij}^{\overline{*}}+O(k^{3}).
       \end{equation}
       Now, expanding the Taylor series about $(x_{i},y_{j},t^{n})$ with mesh size $h$ using central difference representation to get
       \begin{equation*}
        u_{yy,ij}^{n}=\delta_{y}^{2}u^{n}_{ij}+O(h^{2});\text{\,\,}\left(f(u)\right)_{yy,ij}^{n}=\delta_{y}^{2}\left(f(u^{n}_{ij})\right)
        +O(h^{2});\text{\,\,}u_{4y,ij}^{n}=\delta_{y}^{2}(\delta_{y}^{2}u^{n}_{ij})+O(h^{2});
       \end{equation*}
       \begin{equation}\label{14}
        u_{yy,ij}^{\overline{*}}=\delta_{y}^{2}u^{\overline{*}}_{ij}+O(h^{2});\text{\,\,}\left(f(u)\right)_{yy,ij}^{\overline{*}}=\delta_{y}^{2}
        \left(f(u^{\overline{*}}_{ij})\right)+O(h^{2});\text{\,\,}u_{4y,ij}^{\overline{*}}=\delta_{y}^{2}(\delta_{y}^{2}u^{\overline{*}}_{ij})+O(h^{2}),
       \end{equation}
       where $\delta_{y}^{2}w^{l}_{ij}$ is given by relation $(\ref{4}).$ Substituting equations $(\ref{14})$ into equations $(\ref{12})$
       and $(\ref{13})$ to obtain
       \begin{equation}\label{15}
        u_{ij}^{\overline{*}}=u^{n}_{ij}+\frac{k}{2}[a\delta_{y}^{2}u^{n}_{ij}+f(u_{ij}^{n})]+k^{2}\rho_{ij}^{n}+O(k^{3}+kh^{2});
       \end{equation}
       and
       \begin{equation}\label{16}
        u_{ij}^{\overline{\overline{*}}}=u^{n}_{ij}+\frac{k}{2}[a\delta_{y}^{2}u^{\overline{*}}_{ij}+f(u_{ij}^{\overline{*}})]+
        k^{2}\rho_{ij}^{\overline{*}}+O(k^{3}+kh^{2}),
       \end{equation}
       where
       \begin{equation}\label{20}
        \rho_{ij}^{\alpha}=\frac{1}{8}\left[a^{2}\delta_{y}^{2}(\delta_{y}^{2}u^{\alpha}_{ij})+a\delta_{y}^{2}f(u_{ij}^{\alpha})
        +a\delta_{y}^{2}u^{\alpha}_{ij}f^{'}(u_{ij}^{\alpha})+f(u_{ij}^{\alpha})f^{'}(u_{ij}^{\alpha})\right],
       \end{equation}
        where $\alpha=n,\overline{*}.$ The term $f(u_{ij}^{\overline{*}})$ should be expressed as a function of $f(u_{ij}^{n}),$ $f^{'}(u_{ij}^{n})$ and $u_{t,ij}^{n}.$ Applying the Taylor expansion about $(x_{i},y_{j},t^{n})$ with time step $k/2$ using forward difference representation to get
        \begin{equation}\label{17}
        f(u_{ij}^{\overline{*}})=f(u_{ij}^{n})+\frac{k}{2}u_{t,ij}^{n}f^{'}(u_{ij}^{n})+O(k^{2}).
        \end{equation}
        But, it comes from equation $(\ref{11a})$ and relations $(\ref{14})$ that
        \begin{equation}\label{18a}
         u_{t,ij}^{n}=au_{yy,ij}^{n}+f(u_{ij}^{n})=a\delta_{y}^{2}u^{n}_{ij}+f(u_{ij}^{n})+O(h^{2}).
        \end{equation}
        This fact, together with equation $(\ref{17})$ result in
        \begin{equation}\label{18}
        f(u_{ij}^{\overline{*}})=f(u_{ij}^{n})+\frac{k}{2}\left[a\delta_{y}^{2}u^{n}_{ij}+f(u_{ij}^{n})\right]f^{'}(u_{ij}^{n})+O(k^{2}+kh^{2}).
        \end{equation}
        Plugging equations $(\ref{15}),$ $(\ref{16})$ and $(\ref{18}),$ straightforward computations give
        \begin{equation}\label{19}
        u_{ij}^{\overline{\overline{*}}}=u^{n}_{ij}+\frac{k}{2}\left(a\delta_{y}^{2}u^{n}_{ij}+f(u_{ij}^{n})\right)
        +k^{2}(2\rho_{ij}^{n}+\rho_{ij}^{\overline{*}})+O(k^{3}+kh^{2}).
       \end{equation}
        Taking the average of $u_{ij}^{\overline{*}}$ and $u_{ij}^{\overline{\overline{*}}}$ to get
        \begin{equation}\label{20a}
        \frac{u_{ij}^{\overline{*}}+u_{ij}^{\overline{\overline{*}}}}{2}=u^{n}_{ij}+\frac{k}{2}\left(a\delta_{y}^{2}u^{n}_{ij}
        +f(u_{ij}^{n})\right)+\frac{1}{2}k^{2}(3\rho_{ij}^{n}+\rho_{ij}^{\overline{*}})+O(k^{3}+kh^{2}),
       \end{equation}
       where $\rho_{ij}^{\alpha}$ is given by relation $(\ref{20}).$\\

       On the other hand, to define the operator $L_{x}(k),$ we should consider the equation
       \begin{equation}\label{21}
        u_{t}=au_{xx}.
       \end{equation}
       It comes from equation $(\ref{21}),$ that
        \begin{equation}\label{21a}
         u_{2t}=au_{xx,t}=a^{2}u_{4x}.
        \end{equation}
       Applying the Taylor series expansion about $(x_{i},y_{i},t^{*})$ (where $t^{*}\in(t^{n},t^{n+1}),$ is the time used at the beginning of the
       next step in a time-split MacCormack scheme) with mesh size $h$ using central difference representation, we obtain
       \begin{equation}\label{22}
       u_{xx,ij}^{*}=\delta_{x}^{2}u^{*}_{ij}+O(h^{2});\text{\,\,}u_{4x,ij}^{*}=\delta_{x}^{2}(\delta_{x}^{2}u^{*}_{ij})+O(h^{2});\text{\,\,}
       u_{xx,ij}^{\overline{**}}=\delta_{x}^{2}u^{\overline{**}}_{ij}+O(h^{2});\text{\,\,}u_{4x,ij}^{\overline{**}}=\delta_{x}^{2}
       (\delta_{x}^{2}u^{\overline{**}}_{ij})+O(h^{2}),
       \end{equation}
       where $\delta_{x}^{2}u^{l}_{ij}$ is defined by equation $(\ref{4}).$ Also, expanding the Taylor series at the predictor and
       corrector steps about $(x_{i},y_{j},t^{*})$ with time step $k$ using forward difference, it is not difficult to observe that
       \begin{equation}\label{23}
        u_{ij}^{\overline{**}}=u^{*}_{ij}+ku_{t})_{ij}^{*}+\frac{k^{2}}{2}u_{2t})_{ij}^{*}+O(k^{3});
        \text{\,\,}u_{ij}^{\overline{\overline{**}}}=u^{*}_{ij}+ku_{t})_{ij}^{\overline{**}}+\frac{k^{2}}{2}u_{2t})_{ij}^{\overline{**}}+O(k^{3}).
       \end{equation}
        A combination of equations $(\ref{23})$, $(\ref{22})$, $(\ref{21})$ and $(\ref{21a})$ provides
        \begin{equation}\label{25}
        u_{ij}^{\overline{**}}=u^{*}_{ij}+ak\delta_{x}^{2}u^{*}_{ij}+\frac{a^{2}k^{2}}{2}\delta_{x}^{2}(\delta_{x}^{2}u^{*}_{ij})+O(k^{3}+kh^{2});
        \text{\,\,}u_{ij}^{\overline{\overline{**}}}=u^{*}_{ij}+ak\delta_{x}^{2}u^{\overline{**}}_{ij}+\frac{a^{2}k^{2}}{2}
        \delta_{x}^{2}(\delta_{x}^{2}u^{\overline{**}}_{ij})+O(k^{3}+kh^{2}).
       \end{equation}
        In order to obtain a simple expression of $\delta_{x}^{2}u^{\overline{**}}_{ij}$, we should use the first equation in $(\ref{25})$. Tracking the infinitesimal term in this equation, direct computations give
       \begin{equation*}
        \delta_{x}^{2}u^{\overline{**}}_{ij}=\delta_{x}^{2}u^{*}_{ij}+ak\delta_{x}^{2}\left(\delta_{x}^{2}u^{*}_{ij}\right)+\frac{a^{2}k^{2}}{2}
        \delta_{x}^{2}(\delta_{x}^{4}u^{*}_{ij}).
       \end{equation*}
        The truncation of this error term does not compromise the result. This fact, together with relation $(\ref{25})$ yield
        \begin{equation}\label{27}
        u^{\overline{\overline{**}}}_{ij}=u^{*}_{ij}+ak\delta_{x}^{2}u^{*}_{ij}+\frac{3a^{2}k^{2}}{2}\delta_{x}^{2}(\delta_{x}^{2}u^{*}_{ij})+O(k^{3}
        +kh^{2}).
        \end{equation}
        Taking the average of $u_{ij}^{\overline{**}}$ and $u_{ij}^{\overline{\overline{**}}}$, it is not hard to see that
        \begin{equation}\label{31}
        \frac{u_{ij}^{\overline{**}}+u_{ij}^{\overline{\overline{**}}}}{2}=u^{*}_{ij}+ak\delta_{x}^{2}u^{*}_{ij}+a^{2}k^{2}
        \delta_{x}^{2}(\delta_{x}^{2}u^{*}_{ij})+O(k^{3}+kh^{2}).
       \end{equation}

        In way similar, starting with the one-dimensional equation: $u_{t}-au_{yy}=f(u)$, expanding the Taylor series about $(x_{i},y_{j},t^{**})$
        (where $t^{**}$ represents the time used at the last step in a time-split MacCormack approach) at the predictor and corrector steps with time
        step $k/2$ and mesh size $h,$ using forward difference representations to get
        \begin{equation}\label{33}
        \frac{u_{ij}^{\overline{n+1}}+u_{ij}^{\overline{\overline{n+1}}}}{2}=u^{**}_{ij}+\frac{k}{2}\left(a\delta_{y}^{2}u^{**}_{ij}
        +f(u_{ij}^{**})\right)+\frac{k^{2}}{2}(3\gamma_{ij}^{**}+\gamma_{ij}^{\overline{n+1}})+O(k^{3}+kh^{2}),
       \end{equation}
       where we set $\mu=**,\overline{n+1}$, and
       \begin{equation}\label{34}
        \gamma_{ij}^{\mu}=\frac{1}{8}\left[a^{2}\delta_{y}^{2}(\delta_{y}^{2}u^{\mu}_{ij})+a\delta_{y}^{2}f(u_{ij}^{\mu})
        +a\delta_{y}^{2}u^{\mu}_{ij}f^{'}(u_{ij}^{\mu})+f(u_{ij}^{\mu})f^{'}(u_{ij}^{\mu})\right].
       \end{equation}

       To construct a three-level explicit time-split MacCormack method for solving the nonlinear reaction-diffusion equation $(\ref{1})$-$(\ref{3}),$
       we must follow the ideas presented in the literature to construct the explicit MacCormack scheme\cite{mc1,mc2,mc4,mc5}. Specifically, we should neglect the terms of second order together with the infinitesimal term $O(k^{3}+kh^{2})$ in equations $(\ref{20a}),$ $(\ref{31})$ and $(\ref{33})$. In addition, the terms $u_{ij}^{*},$ $u_{ij}^{**}$ and $u^{n+1}_{ij}$ must be defined as the average of predicted and corrected values, that is,
       \begin{equation}\label{35}
        u_{ij}^{*}=\frac{u_{ij}^{\overline{*}}+u_{ij}^{\overline{\overline{*}}}}{2};\text{\,\,\,}u_{ij}^{**}=
        \frac{u_{ij}^{\overline{**}}+u_{ij}^{\overline{\overline{**}}}}{2}\text{\,\,\,\,\,\,and\,\,\,\,\,\,}u_{ij}^{n+1}=
        \frac{u_{ij}^{\overline{n+1}}+u_{ij}^{\overline{\overline{n+1}}}}{2}.
       \end{equation}
       Thus, equations
       \begin{equation}\label{36}
        u_{ij}^{*}=L_{y}(k/2)u_{ij}^{n};\text{\,\,\,}u_{ij}^{**}=L_{x}(k)u_{ij}^{*}\text{\,\,\,\,\,\,and\,\,\,\,\,\,}u_{ij}^{n+1}=L_{y}(k/2)u_{ij}^{**},
       \end{equation}
       are by definition equivalent to
       \begin{equation}\label{37}
       u_{ij}^{*}=u^{n}_{ij}+\frac{k}{2}\left(a\delta_{y}^{2}u^{n}_{ij}+f(u_{ij}^{n})\right);\text{\,\,\,}u_{ij}^{**}=u^{*}_{ij}+ak\delta_{x}^{2}u^{*}_{ij}
        \text{\,\,\,\,and\,\,\,\,}u_{ij}^{n+1}=u^{**}_{ij}+\frac{k}{2}\left(a\delta_{y}^{2}u^{**}_{ij}+f(u_{ij}^{**})\right).
       \end{equation}
       Since the operator $L_{y}(k/2)L_{x}(k)L_{y}(k/2)$ is symmetric, this fact together with relations $(\ref{20a})$, $(\ref{31})$ and
        $(\ref{33})$ show that the obtained method is a three-level technique, an explicit predictor-corrector scheme, second order accurate
       in time and fourth order convergent in space. This theoretical result is confirmed by a wide set of numerical examples (we refer
       the readers to section $\ref{sec5}$). From the definition of the linear operators $"\delta_{x}^{2}"$ and $"\delta_{y}^{2}"$ given by
       $(\ref{4})$, equation $(\ref{37})$ can be rewritten as follows. For $n=0,1,...,N-1;$
       \begin{equation}\label{40}
        u_{ij}^{*}=u^{n}_{ij}+\frac{k}{2}\left(\frac{a}{h^{2}}(u^{n}_{i,j+1}-2u^{n}_{ij}+u^{n}_{i,j-1})+f(u_{ij}^{n})\right),\text{\,\,\,\,}
        i=0,1,...,M;\text{\,\,\,}j=1,2,...,M-1;
       \end{equation}
       \begin{equation}\label{41}
        u_{ij}^{**}=u^{*}_{ij}+\frac{ak}{h^{2}}(u^{*}_{i+1,j}-2u^{*}_{ij}+u^{*}_{i-1,j}),
        \text{\,\,\,\,}i=1,2,...,M-1; \text{\,\,\,}j=0,1,...,M;
       \end{equation}
       \begin{equation}\label{42}
        u_{ij}^{n+1}=u^{**}_{ij}+\frac{k}{2}\left(\frac{a}{h^{2}}(u^{**}_{i,j+1}-2u^{**}_{ij}+u^{**}_{i,j-1})+f(u_{ij}^{**})\right),
        \text{\,\,\,\,}i=0,1,...,M;\text{\,\,\,}j=1,2,...,M-1,
       \end{equation}
        with the initial and boundary conditions. For $i,j=0,1,...,M,$
        \begin{equation*}
        u_{ij}^{0}=u_{0}(x_{i},y_{j});\text{\,}u_{i0}^{n}=\varphi^{n}_{i0};\text{\,}u_{iM}^{n}=\varphi^{n}_{iM};
        \text{\,}u_{0j}^{n}=\varphi^{n}_{0j};\text{\,}u_{Mj}^{n}=\varphi^{n}_{Mj};\text{\,\,}u_{0j}^{*}=\varphi^{n+1}_{0j};\text{\,}u_{Mj}^{*}=
        \varphi^{n+1}_{Mj}; \text{\,}u_{j0}^{*}=\varphi^{n+1}_{j0};
        \end{equation*}
        \begin{equation*}
        u_{jM}^{*}=\varphi^{n+1}_{jM};\text{\,\,\,}u_{0j}^{**}=\varphi^{n+1}_{0j};\text{\,}u_{Mj}^{**}=\varphi^{n+1}_{Mj};\text{\,}
        u_{j0}^{**}=\varphi^{n+1}_{j0};\text{\,}u_{jM}^{**}=\varphi^{n+1}_{jM};\text{\,}u_{i0}^{N}=\varphi^{N}_{i0};\text{\,}u_{iM}^{N}=\varphi^{N}_{iM};
        \end{equation*}
        \begin{equation}\label{24}
         u_{0j}^{N}=\varphi^{N}_{0j};\text{\,}u_{Mj}^{N}=\varphi^{N}_{Mj},
        \end{equation}
       which represent a detailed description of a three-level explicit time-split MacCormack method applied to problem $(\ref{1})$-$(\ref{3}).$\\

       In the rest of this paper, we prove the stability, the error estimates and the convergence rate of a three-level time-split MacCormack
       approach under the time step restriction
       \begin{equation}\label{44}
        \frac{2ak}{h^{2}}\leq1,
       \end{equation}
       where $a$ is the diffusive coefficient given in equation $(\ref{1}).$ We recall that $k$ is the time step and $h$ is the grid size.
       Estimate $(\ref{44})$ is well known in literature as CFL condition for stability of the explicit schemes when solving linear parabolic
       equations. We assume that the analytical solution $\overline{u}\in L^{\infty}(0,T;L^{2}(\Omega)) \cap H^{1}(0,T;H^{3}(\Omega))\cap H^{2}(0,T;H^{1}(\Omega))\cap H^{2}(0,T;L^{2}(\Omega))\cap L^{2}(0,T;H^{4}(\Omega)),$ that is, there exists a positive constant $\widetilde{C},$ independent of both time step $k$ and mesh size $h,$ so that
      \begin{equation}\label{43}
      \||\overline{u}|\|_{L^{\infty}(0,T;L^{2}(\Omega))}+\||\overline{u}|\|_{H^{1}(0,T;H^{3}(\Omega))}+\||\overline{u}|\|_{H^{2}(0,T;H^{1}(\Omega))}+
      \||\overline{u}|\|_{H^{2}(0,T;L^{2}(\Omega))}+\||\overline{u}|\|_{L^{2}(0,T;H^{4}(\Omega))}\leq\widetilde{C}.
      \end{equation}

      \section{Stability analysis of a three-level time-split MacCormack scheme}\label{sec3}
      This section considers a deep analysis of the stability of a three-level time-split MacCormack scheme $(\ref{40})$-$(\ref{24})$ for
      solving equations $(\ref{1})$-$(\ref{3}).$

      \begin{theorem}\label{t1}
      Let $u$ be the solution provided by the numerical scheme $(\ref{40})$-$(\ref{24})$. Under the time step restriction $(\ref{44}),$
      the following estimate holds
      \begin{equation*}
      \underset{0\leq n\leq N}{\max}\|u^{n}\|_{L^{2}(\Omega)}\leq\widetilde{C}+\exp\left(CT\underset{l=0}{\overset{3}\sum}(Ck)^{l}\right),
      \end{equation*}
      where $C$ is a positive constant independent of the time step $k$ and mesh size $h$ and $\widetilde{C}$ is
      given by relation $(\ref{43}).$
      \end{theorem}

      The following result (namely Lemmas $\ref{l1}$) plays a crucial role in the proof of Theorem $\ref{t1}$.

       \begin{lemma}\label{l1}
       Setting $u_{ij}^{n}=u(x_{i},y_{j},t^{n})$ be the numerical solution provided by the scheme $(\ref{40})$-$(\ref{24})$, $\overline{u}_{ij}^{n}=\overline{u}(x_{i},y_{j},t^{n})$ be the exact one and let $e_{ij}^{n}=u_{ij}^{n}-\overline{u}_{ij}^{n}$ be the error. We recall that $\overline{u}_{ij}^{*}=\frac{\overline{u}_{ij}^{\overline{*}}+\overline{u}_{ij}^{\overline{\overline{*}}}}{2},$
       $\overline{u}_{ij}^{**}=\frac{\overline{u}_{ij}^{\overline{**}}+\overline{u}_{ij}^{\overline{\overline{**}}}}{2}$,
       satisfy relations $(\ref{20a})$ and $(\ref{31})$, respectively. $u_{ij}^{*}$ and $u_{ij}^{**}$ are given by equations $(\ref{40})$
       and $(\ref{41})$, respectively. The following equalities hold:
       \begin{equation}\label{72}
        a<\delta_{x}^{2}e_{ij}^{n},e_{ij}^{n}>_{x}=h^{2}\underset{j,i=1}{\overset{M-1}\sum}\frac{a}{h^{2}}\left(e_{i+1,j}^{n}-2e_{ij}^{n}
        +e_{i-1,j}^{n}\right)e_{ij}^{n}=-a\|\delta_{x}e^{n}\|_{L^{2}(\Omega)}^{2},
       \end{equation}
       and
       \begin{equation}\label{72a}
        a<\delta_{y}^{2}e_{ij}^{n},e_{ij}^{n}>_{y}=h^{2}\underset{j,i=1}{\overset{M-1}\sum}\frac{a}{h^{2}}\left(e_{i,j+1}^{n}-2e_{ij}^{n}
        +e_{i,j-1}^{n}\right)e_{ij}^{n}=-a\|\delta_{y}e^{n}\|_{L^{2}(\Omega)}^{2},
       \end{equation}
       where the operators $\delta_{x}$ and $\delta_{y}$ are given by relation $(\ref{4}).$
       \end{lemma}

       \begin{proof}(of Lemma $\ref{l1}$).  Firstly, it is not hard to observe that
       \begin{equation*}
       \frac{a}{h^{2}}\left(e_{i,j+1}^{n}-2e_{ij}^{n}+e_{i,j-1}^{n}\right)e_{ij}^{n}=\frac{a}{h}(\delta_{y}e_{i,j+\frac{1}{2}}^{n}-\delta_{y}
       e_{i,j-\frac{1}{2}}^{n})e_{ij}^{n},
       \end{equation*}
       for $i=0,1,...,M$ and $j=1,2,...,M-1.$ We should prove only equation $(\ref{72}).$ The proof of relation $(\ref{72a})$ is similar.\\

       It follows from the definition of the operator $\delta_{x}^{2}$ and the scalar product $<\cdot,\cdot>_{x}$ given by
       $(\ref{4})$ and $(\ref{6}),$ respectively, that
       \begin{equation*}
        a<\delta_{x}^{2}e_{ij}^{n},e_{ij}^{n}>_{x}=h^{2}\underset{i,j=1}{\overset{M-1}\sum}\frac{a}{h^{2}}\left(e_{i+1,j}^{n}-2e_{ij}^{n}
        +e_{i-1,j}^{n}\right)e_{ij}^{n}=a\underset{i,j=1}{\overset{M-1}\sum}\left[\left(e_{i+1,j}^{n}-e_{ij}^{n}\right)e_{ij}^{n}-
        \left(e_{i,j}^{n}-e_{i-1,j}^{n}\right)e_{ij}^{n}\right]=
       \end{equation*}
       \begin{equation*}
        ah\underset{i,j=1}{\overset{M-1}\sum}\left[\left(\frac{e_{i+1,j}^{n}-e_{ij}^{n}}{h}\right)e_{ij}^{n}-\left(\frac{e_{i,j}^{n}-e_{i-1,j}^{n}}{h}\right)
        e_{ij}^{n}\right]=ah\underset{i,j=1}{\overset{M-1}\sum}\left[\left(\delta_{x}e_{i+\frac{1}{2},j}^{n}\right)e_{ij}^{n}-
        \left(\delta_{x}e_{i-\frac{1}{2},j}^{n}\right)e_{ij}^{n}\right]=
       \end{equation*}
       \begin{equation*}
        ah\underset{j=1}{\overset{M-1}\sum}\left\{\left((\delta_{x}e_{\frac{3}{2},j}^{n})e_{1j}^{n}-(\delta_{x}e_{\frac{1}{2},j}^{n})e_{1j}^{n}\right)+
        \left((\delta_{x}e_{\frac{5}{2},j}^{n})e_{2j}^{n}-(\delta_{x}e_{\frac{3}{2},j}^{n})e_{2j}^{n}\right)+\left((\delta_{x}e_{\frac{7}{2},j}^{n})e_{3j}^{n}-
        (\delta_{x}e_{\frac{5}{2},j}^{n})e_{5j}^{n}\right)+\cdots+\right.
       \end{equation*}
       \begin{equation*}
        \left.\left((\delta_{x}e_{M-\frac{3}{2},j}^{n})e_{M-2,j}^{n}-(\delta_{x}e_{M-\frac{5}{2},j}^{n})e_{M-2,j}^{n}\right)+
        \left((\delta_{x}e_{M-\frac{1}{2},j}^{n})e_{M-1,j}^{n}-(\delta_{x}e_{M-\frac{3}{2},j}^{n})e_{M-1,j}^{n}\right)\right\}=
       \end{equation*}
       \begin{equation*}
       ah\underset{j=1}{\overset{M-1}\sum}\left\{-\left(e_{2j}^{n}-e_{1j}^{n}\right)\delta_{x}e_{\frac{3}{2},j}^{n}-
        \left(e_{3j}^{n}-e_{2j}^{n}\right)\delta_{x}e_{\frac{5}{2},j}^{n}-\left(e_{4j}^{n}-e_{3j}^{n}\right)\delta_{x}e_{\frac{7}{2},j}^{n}-\cdots-\right.
       \end{equation*}
       \begin{equation}\label{73}
        \left.\left(e_{M-1,j}^{n}-e_{M-2,j}^{n}\right)\delta_{x}e_{M-\frac{3}{2},j}^{n}+(\delta_{x}e_{M-\frac{1}{2},j}^{n})e_{M-1,j}^{n}-
        (\delta_{x}e_{\frac{1}{2},j}^{n})e_{1,j}^{n}\right\}.
       \end{equation}
       It comes from the boundary condition $(\ref{24})$ that $e_{Mj}^{n}=e_{0j}^{n}=0.$ So $(\delta_{x}e_{M-\frac{1}{2},j}^{n})e_{Mj}^{n}=0$
       and $(\delta_{x}e_{\frac{1}{2},j}^{n})e_{0j}^{n}=0.$ This fact, together with equation $(\ref{73})$ provide
       \begin{equation*}
        h^{2}\underset{i,j=1}{\overset{M-1}\sum}\frac{a}{h^{2}}\left(e_{i+1,j}^{n}-2e_{ij}^{n}+e_{i-1,j}^{n}\right)e_{ij}^{n}=
        ah^{2}\underset{j,i=1}{\overset{M-1}\sum}\left\{-\left(\delta_{x}e_{\frac{3}{2},j}^{n}\right)^{2}-
        \left(\delta_{x}e_{\frac{5}{2},j}^{n}\right)^{2}-\cdots-\left(\delta_{x}e_{M-\frac{3}{2},j}^{n}\right)^{2}  \right.
       \end{equation*}
       \begin{equation*}
        \left.-\left(\frac{e_{Mj}^{n}-e_{M-1,j}^{n}}{h}\right)\delta_{x}e_{M-\frac{1}{2},j}^{n}-\left(\frac{e_{1j}^{n}-e_{0,j}^{n}}{h}\right)
        \delta_{x}e_{\frac{1}{2},j}^{n}\right\}=-ah^{2}\underset{i=0}{\overset{M-1}\sum}{\underset{j=1}{\overset{M-1}\sum}}\left(\delta_{x}
        e_{i+\frac{1}{2},j}^{n}\right)^{2}=-a\|\delta_{x}e^{n}\|_{L^{2}(\Omega)}^{2}.
       \end{equation*}
       \end{proof}

       \begin{proof} (of Theorem \ref{t1}).
        A combination of equations $(\ref{20a})$, $(\ref{35})$ and $(\ref{40})$ gives
       \begin{equation}\label{45}
        e_{ij}^{*}=e^{n}_{ij}+\frac{k}{2}\left(a\delta_{y}^{2}e_{ij}^{n}+f(u_{ij}^{n})-f(\overline{u}_{ij}^{n})\right)
        +\frac{k^{2}}{2}(3\rho_{ij}^{n}+\rho_{ij}^{\overline{*}})+O(k^{3}+kh^{2}),
       \end{equation}
       where $\rho_{ij}^{\alpha}$ is defined by $(\ref{20})$. Utilizing the definition of the operator $"\delta_{y}^{2}",$ equation $(\ref{45})$ is
       equivalent to
       \begin{equation}\label{46}
        e_{ij}^{*}=e^{n}_{ij}+\frac{k}{2}\left(\frac{a}{h^{2}}(e^{n}_{i,j+1}-2e^{n}_{ij}+e^{n}_{i,j-1})+f(u_{ij}^{n})
        -f(\overline{u}_{ij}^{n})\right)+\frac{k^{2}}{2}(3\rho_{ij}^{n}+\rho_{ij}^{\overline{*}})+O(k^{3}+kh^{2}).
       \end{equation}
        Of course, the aim of this study is to give a general picture of the stability analysis of the numerical scheme $(\ref{40})$-$(\ref{24})$. Since the formulas can become quite heavy, for the sake of readability, we should neglect the higher order terms in time step $k$ and grid spacing $h.$
        However, the truncation of the infinitesimal terms does not compromise the result on the stability analysis. Using this, equation $(\ref{46})$ becomes
        \begin{equation*}
        e_{ij}^{*}=e^{n}_{ij}+\frac{k}{2}\left(\frac{a}{h^{2}}(e^{n}_{i,j+1}-2e^{n}_{ij}+e^{n}_{i,j-1})+f(u_{ij}^{n})-f(\overline{u}_{ij}^{n})\right).
       \end{equation*}
        Taking the square, it holds
        \begin{equation*}
        (e_{ij}^{*})^{2}=(e^{n}_{ij})^{2}+k\left(\frac{a}{h^{2}}(e^{n}_{i,j+1}-2e^{n}_{ij}+e^{n}_{i,j-1})+f(u_{ij}^{n})-f(\overline{u}_{ij}^{n})\right)e^{n}_{ij}
       \end{equation*}
       \begin{equation}\label{47}+\frac{k^{2}}{4}\left(\frac{a}{h^{2}}(e^{n}_{i,j+1}-2e^{n}_{ij}+e^{n}_{i,j-1})+f(u_{ij}^{n})
       -f(\overline{u}_{ij}^{n})\right)^{2}.
       \end{equation}
        Now, using equality $2(a-b)b=a^{2}-b^{2}-(a-b)^{2}$ and inequality $(a\pm b)^{2}\leq2(a^{2}+b^{2}),$ for any $a,b\in\mathbb{R},$
        and by simple computations, it is not hard to see that
        \begin{equation}\label{49}
        (e^{n}_{i,j+1}-2e^{n}_{ij}+e^{n}_{i,j-1})^{2}\leq2\left[(e^{n}_{i,j+1}-e^{n}_{ij})^{2}+(e^{n}_{i,j-1}-e^{n}_{ij})^{2}\right];
        \end{equation}
        \begin{equation*}
        \left[\frac{a}{h^{2}}(e^{n}_{i,j+1}-2e^{n}_{ij}+e^{n}_{i,j-1})+f(u_{ij}^{n})-f(\overline{u}_{ij}^{n})\right]^{2}
        \leq2\left[\frac{2a^{2}}{h^{4}}\left[(e^{n}_{i,j+1}-e^{n}_{i,j})^{2}+(e^{n}_{i,j}-e^{n}_{i,j-1})^{2}\right]+\right.
        \end{equation*}
        \begin{equation}\label{50}
        \left.\left(f(u_{ij}^{n})-f(\overline{u}_{ij}^{n})\right)^{2}\right].
        \end{equation}
        $f\in\mathcal{C}^{1}(\mathbb{R})$ is a Lipschitz function, so there is a positive constant $C$ independent of the time step $k$ and
        the mesh size $h$ so that
        \begin{equation}\label{53}
        |f(u_{ij}^{n})-f(\overline{u}_{ij}^{n})|\leq C|e_{ij}^{n}|.
        \end{equation}
        From inequality $(\ref{53})$, it is easy to see that
        \begin{equation}\label{54}
        \left(f(u_{ij}^{n})-f(\overline{u}_{ij}^{n})\right)e_{ij}^{n}\leq C(e_{ij}^{n})^{2}\text{\,\,\,and\,\,\,}\left(f(u_{ij}^{n})-
        f(\overline{u}_{ij}^{n})\right)^{2}\leq C^{2}(e_{ij}^{n})^{2}.
        \end{equation}
        A combination of estimates $(\ref{47})$-$(\ref{54})$ results in
        \begin{equation*}
        (e_{ij}^{*})^{2}\leq (e^{n}_{ij})^{2}+k\left\{\frac{a}{h^{2}}(e^{n}_{i,j+1}-2e^{n}_{ij}+e^{n}_{i,j-1})e^{n}_{ij}+C(e^{n}_{ij})^{2}\right\}
        +\frac{k^{2}}{2}\left\{\frac{2a^{2}}{h^{2}}\left[\left(\delta_{y}e^{n}_{i,j+\frac{1}{2}}\right)^{2}+\left(\delta_{y}e^{n}_{i,j-\frac{1}{2}}
        \right)^{2}\right]+\right.
       \end{equation*}
        \begin{equation}\label{55}
       \left.C^{2}(e^{n}_{ij})^{2}\right\}.
       \end{equation}
       Summing this up from $i,j=1,2,...,M-1,$ and rearranging terms, this provides
       \begin{equation*}
        \underset{i,j=1}{\overset{M-1}\sum}(e_{ij}^{*})^{2}\leq \underset{i,j=1}{\overset{M-1}\sum}(e^{n}_{ij})^{2}+Ck\left(1
        +\frac{Ck}{2}\right)\underset{i,j=1}{\overset{M-1}\sum}(e^{n}_{ij})^{2}+\frac{ak}{h^{2}}
        \underset{i,j=1}{\overset{M-1}\sum}(e^{n}_{i,j+1}-2e^{n}_{ij}+e^{n}_{i,j-1})e^{n}_{ij}+
       \end{equation*}
       \begin{equation*}
       \frac{a^{2}k^{2}}{h^{2}}\underset{i,j=1}{\overset{M-1}\sum}\left[\left(\delta_{y}e^{n}_{i,j+\frac{1}{2}}\right)^{2}
       +\left(\delta_{y}e^{n}_{i,j-\frac{1}{2}}\right)^{2}\right],
       \end{equation*}
       which implies
       \begin{equation*}
        \underset{i,j=1}{\overset{M-1}\sum}(e_{ij}^{*})^{2}\leq \underset{i,j=1}{\overset{M-1}\sum}(e^{n}_{ij})^{2}+Ck\left(1
        +\frac{Ck}{2}\right)\underset{i,j=1}{\overset{M-1}\sum}(e^{n}_{ij})^{2}+\frac{ak}{h^{2}}
        \underset{i,j=1}{\overset{M-1}\sum}(e^{n}_{i,j+1}-2e^{n}_{ij}+e^{n}_{i,j-1})e^{n}_{ij}+
       \end{equation*}
       \begin{equation}\label{56}
       \frac{2a^{2}k^{2}}{h^{2}}\underset{i=1}{\overset{M-1}\sum}{\underset{j=0}{\overset{M-1}\sum}}\left(\delta_{y}e^{n}_{i,j+\frac{1}{2}}\right)^{2}.
       \end{equation}
       Multiplying both sides of inequality $(\ref{56})$ by $h^{2},$ and using equation $(\ref{72a})$ to get
       \begin{equation*}
        \|e^{*}\|_{L^{2}(\Omega)}^{2}\leq \|e^{n}\|_{L^{2}(\Omega)}^{2}+Ck\left(1+\frac{Ck}{2}\right)\|e^{n}\|_{L^{2}(\Omega)}^{2}-
        ak\|\delta_{y}e^{n}\|_{L^{2}(\Omega)}^{2}+\frac{2a^{2}k^{2}}{h^{2}}\|\delta_{y}e^{n}\|_{L^{2}(\Omega)}^{2}.
       \end{equation*}
       From the time step restriction $(\ref{44})$, $-1+\frac{2ak}{h^{2}}\leq0,$ utilizing this, it follows
       \begin{equation}\label{57}
        \|e^{*}\|_{L^{2}(\Omega)}^{2}\leq \|e^{n}\|_{L^{2}(\Omega)}^{2}+Ck\left(1+\frac{Ck}{2}\right)\|e^{n}\|_{L^{2}(\Omega)}^{2}.
       \end{equation}
       In way similar, combining equations $(\ref{33}),$ $(\ref{35})$ and $(\ref{42})$, it is not hard to show that
       \begin{equation}\label{59}
        \|e^{n+1}\|_{L^{2}(\Omega)}^{2}\leq \|e^{**}\|_{L^{2}(\Omega)}^{2}+Ck\left(1+\frac{Ck}{2}\right)\|e^{**}\|_{L^{2}(\Omega)}^{2}.
       \end{equation}
       We must find a similar estimate associated with $\|e^{**}\|_{L^{2}(\Omega)}^{2}$ and $\|e^{*}\|_{L^{2}(\Omega)}^{2}.$ Using equations $(\ref{31}),$ $(\ref{35})$ and $(\ref{41})$, it holds
       \begin{equation*}
        e_{ij}^{**}=e^{*}_{ij}+\frac{ak}{h^{2}}(e^{*}_{i+1,j}-2e^{*}_{ij}+e^{*}_{i-1,j}),
       \end{equation*}
        for $i=1,2,...,M-1,$ and $j=0,1,...,M.$ Taking the square, we obtain
       \begin{equation*}
        (e_{ij}^{**})^{2}=(e^{*}_{ij})^{2}+\frac{2ak}{h^{2}}(e^{*}_{i+1,j}-2e^{*}_{ij}+e^{*}_{i-1,j})e_{ij}^{*}+
        \frac{a^{2}k^{2}}{h^{4}}(e^{*}_{i+1,j}-2e^{*}_{ij}+e^{*}_{i-1,j})^{2},
       \end{equation*}
       which implies
       \begin{equation}\label{59a}
        (e_{ij}^{**})^{2}\leq(e^{*}_{ij})^{2}+\frac{2ak}{h^{2}}(e^{*}_{i+1,j}-2e^{*}_{ij}+e^{*}_{i-1,j})e_{ij}^{*}+
        \frac{2a^{2}k^{2}}{h^{4}}\left[(e^{*}_{i+1,j}-e^{*}_{ij})^{2}+(e^{*}_{i-1,j}-e^{*}_{ij})^{2}\right],
       \end{equation}
        Now, summing relation $(\ref{59a})$ for $i,j=1,2,...,M-1,$ and multiplying the obtained equation by $h^{2},$ this results in
       \begin{equation*}
        h^{2}\underset{i,j=1}{\overset{M-1}\sum}(e_{ij}^{**})^{2}\leq h^{2}\underset{i,j=1}{\overset{M-1}\sum}(e^{*}_{ij})^{2}+
        2ak\underset{i,j=1}{\overset{M-1}\sum}(e^{*}_{i+1,j}-2e^{*}_{ij}+e^{*}_{i-1,j})e_{ij}^{*}+
        2a^{2}k^{2}\underset{i,j=1}{\overset{M-1}\sum}\left[(\delta_{x}e^{*}_{i+\frac{1}{2},j})^{2}+(\delta_{x}e^{*}_{i-\frac{1}{2},j})^{2}\right],
       \end{equation*}
       which implies
       \begin{equation*}
        h^{2}\underset{i,j=1}{\overset{M-1}\sum}(e_{ij}^{**})^{2}\leq h^{2}\underset{i,j=1}{\overset{M-1}\sum}(e^{*}_{ij})^{2}+
        2ak\underset{i,j=1}{\overset{M-1}\sum}(e^{*}_{i+1,j}-2e^{*}_{ij}+e^{*}_{i-1,j})e_{ij}^{*}+
        4a^{2}k^{2}\underset{i=0}{\overset{M-1}\sum}{\underset{j=1}{\overset{M-1}\sum}}(\delta_{x}e^{*}_{i+\frac{1}{2},j})^{2}.
       \end{equation*}
       which is equivalent to
       \begin{equation*}
        \|e^{**}\|_{L^{2}(\Omega)}^{2}\leq\|e^{*}\|_{L^{2}(\Omega)}^{2}+2ak\underset{i,j=1}{\overset{M-1}\sum}
        (e^{*}_{i+1,j}-2e^{*}_{ij}+e^{*}_{i-1,j})e_{ij}^{*}+4a^{2}k^{2}h^{-2}\|\delta_{x}e^{*}\|_{L^{2}(\Omega)}^{2}.
       \end{equation*}
        Utilizing equality $(\ref{72}),$ this gives
        \begin{equation}\label{59b}
        \|e^{**}\|_{L^{2}(\Omega)}^{2}\leq\|e^{*}\|_{L^{2}(\Omega)}^{2}+2ak\left(-1+\frac{2ak}{h^{2}}\right)\|\delta_{x}e^{*}\|_{L^{2}(\Omega)}^{2}.
       \end{equation}
        It comes from the time step restriction $(\ref{44})$ that $\frac{2ak}{h^{2}}\leq1.$ So, estimate $(\ref{59b})$ provides
        \begin{equation}\label{59c}
        \|e^{**}\|_{L^{2}(\Omega)}^{2}\leq\|e^{*}\|_{L^{2}(\Omega)}^{2}.
       \end{equation}
       Now, plugging inequalities $(\ref{57})$, $(\ref{59})$ and $(\ref{59c}),$ straightforward calculations yield
       \begin{equation*}
        \|e^{n+1}\|_{L^{2}(\Omega)}^{2}\leq \left[1+Ck\left(1+\frac{Ck}{2}\right)\right]^{2}\|e^{n}\|_{L^{2}(\Omega)}^{2}=
        \|e^{n}\|_{L^{2}(\Omega)}^{2}+Ck\left[2+2Ck+C^{2}k^{2}+\frac{1}{4}C^{3}k^{3}\right]\|e^{n}\|_{L^{2}(\Omega)}^{2}.
       \end{equation*}
       Summing this up from $n=0,1,2,..,p-1,$ for any nonnegative integer $p$ satisfying $1\leq p\leq N,$ to get
       \begin{equation}\label{60}
        \|e^{p}\|_{L^{2}(\Omega)}^{2}\leq\|e^{0}\|_{L^{2}(\Omega)}^{2}+Ck\left[2+2Ck+C^{2}k^{2}+\frac{1}{4}C^{3}k^{3}\right]
        \underset{n=0}{\overset{p-1}\sum}\|e^{n}\|_{L^{2}(\Omega)}^{2}.
       \end{equation}
       It comes from the initial condition given in $(\ref{24})$, that $e^{0}_{ij}=0,$ for $0\leq i,j\leq M.$ Applying the discrete Gronwall
       Lemma, estimate $(\ref{60})$ gives
       \begin{equation}\label{61}
        \|e^{p}\|_{L^{2}(\Omega)}^{2}\leq\exp\left\{Ckp\left(2+2Ck+C^{2}k^{2}+\frac{1}{4}C^{3}k^{3}\right)\right\}.
       \end{equation}
       But $k=\frac{T}{N},$ so $Ckp=CT\frac{p}{N}\leq CT$ (since $p\leq N$). This fact, together with estimate $(\ref{61})$ result in
       \begin{equation*}
        \|e^{p}\|_{L^{2}(\Omega)}^{2}\leq\exp\left\{CT\left(2+2Ck+C^{2}k^{2}+\frac{1}{4}C^{3}k^{3}\right)\right\}.
       \end{equation*}
       Taking the square root, it is easy to see that
       \begin{equation}\label{62}
        \|e^{p}\|_{L^{2}(\Omega)}\leq\exp\left\{CT\left(1+Ck+\frac{1}{2}C^{2}k^{2}+\frac{1}{8}C^{3}k^{3}\right)\right\}.
       \end{equation}
       We have that $\|u^{p}\|_{L^{2}(\Omega)}-\|\overline{u}^{p}\|_{L^{2}(\Omega)}\leq\|u^{p}-\overline{u}^{p}\|_{L^{2}(\Omega)}=
       \|e^{p}\|_{L^{2}(\Omega)}.$ A combination of this inequality together with relation $(\ref{62})$ yields
       \begin{equation*}
        \|u^{p}\|_{L^{2}(\Omega)}\leq\|\overline{u}^{p}\|_{L^{2}(\Omega)}+\exp\left\{CT\left(1+Ck+\frac{1}{2}C^{2}k^{2}+\frac{1}{8}C^{3}k^{3}
        \right)\right\}\leq
       \end{equation*}
       \begin{equation*}
        \|\overline{u}^{p}\|_{L^{2}(\Omega)}+\exp\left\{CT\underset{l=0}{\overset{3}\sum}(Ck)^{l}\right\}.
       \end{equation*}
       Since $\overline{u}$ is the exact solution, the proof of Theorem $\ref{t1}$ is completed thanks to estimate $(\ref{43}).$
      \end{proof}

      \section{Convergence of the method}\label{sec4}

      This section deals with the error estimates of a three-level time-split MacCormack method $(\ref{40})$-$(\ref{24})$ applied to equations $(\ref{1})$-$(\ref{3}),$ under the time step restriction $(\ref{44}).$ We assume that the exact solution $\overline{u}$ satisfies estimate $(\ref{43})$. We recall that
      \begin{equation}\label{65}
      \mathcal{U}_{h}=\{u_{ij}^{n},\text{\,}n=0,1,2,...,N;\text{\,}i,j=0,1,2,...,M\},
      \end{equation}
      is the space of grid functions defined on $\Omega_{h}\times\Omega_{k},$ where $\Omega_{k}=\{t^{n},\text{\,}0\leq n\leq N\}$ and\\ $\Omega_{h}=\{(x_{i},y_{j}),\text{\,}0\leq i,j\leq M\}\cap\Omega.$\\

      Let introduce the following discrete norms
      \begin{equation*}
        \||u|\|_{L^{\infty}(0,T;L^{2}(\Omega))}=\underset{0\leq n\leq N}{\max}\|u^{n}\|_{L^{2}(\Omega)};\text{\,\,\,\,\,\,\,}
      \||u|\|_{L^{2}(0,T;L^{2}(\Omega))}=\left(k\underset{n=0}{\overset{N}\sum}\|u^{n}\|_{L^{2}(\Omega)}^{2}\right)^{\frac{1}{2}},
      \end{equation*}
      and
      \begin{equation}\label{66}
      \||u|\|_{L^{1}(0,T;L^{2}(\Omega))}=k\underset{n=0}{\overset{N}\sum}\|u^{n}\|_{L^{2}(\Omega)};\text{\,\,\,for\,\,\,}u\in\mathcal{U}_{h}.
      \end{equation}

      \begin{theorem}\label{t2}
       Suppose $u$ be the solution provided by a three-level time-split MacCormack approach $(\ref{40})$-$(\ref{24})$.
       Under the time step restriction $(\ref{44}),$ the error term $e=u-\overline{u},$ satisfies
      \begin{equation*}
      \||e|\|_{L^{\infty}(0,T;L^{2}(\Omega))}\leq O(k+kh^{2})=O(k+h^{4}).
      \end{equation*}
      \end{theorem}
         The proof of Theorem $\ref{t2}$ requires some intermediate results (namely Lemmas $\ref{l1}$, $\ref{l2}$ and $\ref{l3}$).
       \begin{lemma}\label{l2}
       Consider $v\in H^{4}(\Omega),$ be a function satisfying $v|_{[x_{i},x_{i+1}]}\in\mathcal{C}^{6}[x_{i},x_{i+1}],$ for $i=0,1,2,...,M-1.$
       Then, it holds
       \begin{equation*}
        \frac{1}{h^{2}}(v_{i+1}-2v_{i}+v_{i-1})-v_{2x,i}=\frac{h^{2}}{12}v_{4x,i}-\frac{h^{4}}{720}\left[v_{6x}(\theta^{(3)}_{i})+v_{6x}
        (\theta^{(4)}_{i})\right],\text{\,\,\,\,for\,\,\,\,}i=1,2,...,M-1,
       \end{equation*}
       where $\theta^{(4)}_{i}\in (x_{i-1},x_{i})$, $\theta^{(3)}_{i}\in (x_{i},x_{i+1})$ and $v_{mx}$ denotes the derivative of order $m$ of $v.$ Furthermore,
       for $i=2,3,...,M-2,$
       \begin{equation*}
        \frac{1}{h^{4}}(v_{i+2}-4v_{i+1}+6v_{i}-4v_{i-1}+v_{i-2})-v_{4x,i}=h^{2}\left\{\frac{1}{720}\left[v_{6x}(\theta^{(1)}_{i})+v_{6x}
        (\theta^{(2)}_{i})\right]+\frac{241}{3220}\left[v_{6x}(\theta^{(3)}_{i})+v_{6x}(\theta^{(4)}_{i})\right]\right\},
       \end{equation*}
       where $\theta^{(2)}_{i}\in (x_{i-2},x_{i-1})$, $\theta^{(4)}_{i}\in (x_{i-1},x_{i}),$ $\theta^{(3)}_{i}\in (x_{i},x_{i+1})$ and
       $\theta^{(1)}_{i}\in (x_{i+1},x_{i+2}).$
       \end{lemma}

       \begin{proof}(of Lemma $\ref{l2}$)
       Expanding the Taylor series about $x_{i}$ with grid spacing $h$ using both forward and backward differences to obtain
       \begin{equation}\label{76}
        v_{i+2}=v_{i+1}+hv_{x,i+1}+\frac{h^{2}}{2}v_{2x,i+1}+\frac{h^{3}}{6}v_{3x,i+1}+\frac{h^{4}}{24}v_{4x,i+1}+\frac{h^{5}}{120}v_{5x,i+1}
        +\frac{h^{6}}{720}v_{6x}(\theta^{(1)}_{i}),
       \end{equation}
       where $\theta^{(1)}_{i}\in (x_{i+1},x_{i+2});$
       \begin{equation}\label{77}
        v_{i-2}=v_{i-1}-hv_{x,i-1}+\frac{h^{2}}{2}v_{2x,i-1}-\frac{h^{3}}{6}v_{3x,i-1}+\frac{h^{4}}{24}v_{4x,i-1}-\frac{h^{5}}{120}v_{5x,i-1}
        +\frac{h^{6}}{720}v_{6x}(\theta^{(2)}_{i}),
       \end{equation}
       where $\theta^{(2)}_{i}\in (x_{i-2},x_{i-1});$
       \begin{equation}\label{78}
        v_{i}=v_{i+1}-hv_{x,i+1}+\frac{h^{2}}{2}v_{2x,i+1}-\frac{h^{3}}{6}v_{3x,i+1}+\frac{h^{4}}{24}v_{4x,i+1}-\frac{h^{5}}{120}v_{5x,i+1}
        +\frac{h^{6}}{720}v_{6x}(\theta^{(3)}_{i}),
       \end{equation}
       where $\theta^{(3)}_{i}\in (x_{i},x_{i+1});$
       \begin{equation}\label{79}
        v_{i}=v_{i-1}+hv_{x,i-1}+\frac{h^{2}}{2}v_{2x,i-1}+\frac{h^{3}}{6}v_{3x,i-1}+\frac{h^{4}}{24}v_{4x,i-1}+\frac{h^{5}}{120}v_{5x,i-1}
        +\frac{h^{6}}{720}v_{6x}(\theta^{(4)}_{i}),
       \end{equation}
       where $\theta^{(4)}_{i}\in (x_{i-1},x_{i}).$\\

       In way similar, applying the Taylor expansion for both derivative and higher order derivatives of $v$ to obtain
       \begin{equation}\label{104}
        v_{x,i+1}=v_{x,i}+hv_{2x,i}+\frac{h^{2}}{2}v_{3x,i}+\frac{h^{3}}{6}v_{4x,i}+\frac{h^{4}}{24}v_{5x,i}+\frac{h^{5}}{120}v_{6x}(\theta^{(5)}_{i}),
       \end{equation}
       where $\theta^{(5)}_{i}\in (x_{i},x_{i+1});$
       \begin{equation}\label{105}
        v_{x,i-1}=v_{x,i}-hv_{2x,i}+\frac{h^{2}}{2}v_{3x,i}-\frac{h^{3}}{6}v_{4x,i}+\frac{h^{4}}{24}v_{5x,i}-\frac{h^{5}}{120}v_{6x}(\theta^{(6)}_{i}),
       \end{equation}
       where $\theta^{(6)}_{i}\in (x_{i-1},x_{i});$
       \begin{equation}\label{106}
        v_{2x,i+1}=v_{2x,i}+hv_{3x,i}+\frac{h^{2}}{2}v_{4x,i}+\frac{h^{3}}{6}v_{5x,i}+\frac{h^{4}}{24}v_{6x}(\theta^{(7)}_{i}),
       \end{equation}
       where $\theta^{(7)}_{i}\in(x_{i},x_{i+1});$
       \begin{equation}\label{107}
        v_{2x,i-1}=v_{2x,i}-hv_{3x,i}+\frac{h^{2}}{2}v_{4x,i}-\frac{h^{3}}{6}v_{5x,i}+\frac{h^{4}}{24}v_{6x}(\theta^{(8)}_{i}),
       \end{equation}
       where $\theta^{(8)}_{i}\in (x_{i-1},x_{i});$
       \begin{equation}\label{108}
        v_{3x,i+1}=v_{3x,i}+hv_{4x,i}+\frac{h^{2}}{2}v_{5x,i}+\frac{h^{3}}{6}v_{6x}(\theta^{(9)}_{i}),\text{\,\,\,}v_{3x,i-1}=v_{3x,i}-hv_{4x,i}
        +\frac{h^{2}}{2}v_{5x,i}-\frac{h^{3}}{6}v_{6x}(\theta^{(10)}_{i}),
       \end{equation}
       where $\theta^{(9)}_{i}\in (x_{i},x_{i+1}),$ $\theta^{(10)}_{i}\in (x_{i-1},x_{i});$
       \begin{equation}\label{110}
        v_{4x,i+1}=v_{4x,i}+hv_{5x,i}+\frac{h^{2}}{2}v_{6x}(\theta^{(11)}_{i}),\text{\,\,\,}v_{4x,i-1}=v_{4x,i}-hv_{5x,i}
        +\frac{h^{2}}{2}v_{6x}(\theta^{(12)}_{i}),
       \end{equation}
       where $\theta^{(11)}_{i}\in (x_{i},x_{i+1}),$ $\theta^{(12)}_{i}\in (x_{i-1},x_{i});$
       \begin{equation}\label{111}
        v_{5x,i+1}=v_{5x,i}+hv_{6x}(\theta^{(13)}_{i}),\text{\,\,\,}v_{5x,i-1}=v_{5x,i}-hv_{6x}(\theta^{(12)}_{i}),
       \end{equation}
       where $\theta^{(13)}_{i}\in (x_{i},x_{i+1}),$ $\theta^{(14)}_{i}\in (x_{i-1},x_{i}).$\\

       Now, adding equations $(\ref{78})$-$(\ref{79})$ side by side, this gives
       \begin{equation*}
        2v_{i}=v_{i+1}+v_{i-1}-h(v_{x,i+1}-v_{x,i-1})+\frac{h^{2}}{2}(v_{2x,i+1}+v_{2x,i-1})-\frac{h^{3}}{6}(v_{3x,i+1}-v_{3x,i-1})+
       \end{equation*}
       \begin{equation}\label{109}
        \frac{h^{4}}{24}(v_{4x,i+1}+v_{4x,i-1})-\frac{h^{5}}{120}(v_{5x,i+1}-v_{5x,i-1})+\frac{h^{6}}{720}(v_{6x}(\theta^{(3)}_{i})
        +v_{6x}(\theta^{(4)}_{i})).
       \end{equation}
        Subtracting $(\ref{105})$ from $(\ref{104})$ and adding side by side $(\ref{106})$ and $(\ref{107}),$ using also equations $(\ref{108}),$ $(\ref{110})$ and $(\ref{112}),$ simple calculations provide
       \begin{equation}\label{114}
        v_{x,i+1}-v_{x,i-1}=2hv_{2x,i}+\frac{h^{3}}{3}v_{4x,i}+\frac{h^{5}}{720}\left(v_{6x}(\theta^{(5)}_{i})+v_{6x}(\theta^{(6)}_{i})\right);
       \end{equation}
       \begin{equation}\label{115}
        v_{2x,i+1}+v_{2x,i-1}=2v_{2x,i}+h^{2}v_{4x,i}+\frac{h^{4}}{24}\left(v_{6x}(\theta^{(7)}_{i})+v_{6x}(\theta^{(8)}_{i})\right);
       \end{equation}
       \begin{equation}\label{116}
        v_{3x,i+1}-v_{3x,i-1}=2hv_{4x,i}+\frac{h^{3}}{6}\left(v_{6x}(\theta^{(9)}_{i})+v_{6x}(\theta^{(10)}_{i})\right);
       \end{equation}
       \begin{equation}\label{117}
       v_{4x,i+1}+v_{4x,i-1}=2v_{4x,i}+\frac{h^{2}}{2}\left(v_{6x}(\theta^{(11)}_{i})+v_{6x}(\theta^{(12)}_{i})\right);\text{\,\,}
        v_{5x,i+1}-v_{5x,i-1}=h\left(v_{6x}(\theta^{(13)}_{i})+v_{6x}(\theta^{(14)}_{i})\right).
       \end{equation}
       Combining equations $(\ref{104})$-$(\ref{117})$, straightforward computations result in
       \begin{equation*}
        2v_{i}=v_{i+1}+v_{i-1}-h^{2}v_{2x,i}-\frac{h^{4}}{12}v_{4x,i}+h^{6}\left\{\frac{1}{720}(v_{6x}(\theta^{(3)}_{i})+v_{6x}(\theta^{(4)}_{i}))-
        \frac{1}{120}(v_{6x}(\theta^{(5)}_{i})+v_{6x}(\theta^{(6)}_{i}))+\right.
       \end{equation*}
       \begin{equation*}
        \frac{1}{48}(v_{6x}(\theta^{(7)}_{i})+v_{6x}(\theta^{(8)}_{i}))-\frac{1}{36}(v_{6x}(\theta^{(9)}_{i})+v_{6x}(\theta^{(10)}_{i}))+
        \frac{1}{48}(v_{6x}(\theta^{(11)}_{i})+v_{6x}(\theta^{(12)}_{i}))-
       \end{equation*}
       \begin{equation}\label{112}
        \left.\frac{1}{120}(v_{6x}(\theta^{(13)}_{i})+v_{6x}(\theta^{(14)}_{i}))\right\}.
       \end{equation}
       Since $\theta^{(3)}_{i},\theta^{(5)}_{i},\theta^{(7)}_{i},\theta^{(9)}_{i},\theta^{(11)}_{i},\theta^{(13)}_{i}\in(x_{i},x_{i+1})$ and $\theta^{(4)}_{i},\theta^{(6)}_{i},\theta^{(8)}_{i},\theta^{(10)}_{i},\theta^{(12)}_{i},\theta^{(14)}_{i}\in(x_{i-1},x_{i}),$
       without loss of generality, we can assume that $\theta^{(3)}_{i}=\theta^{(5)}_{i}=\theta^{(7)}_{i}=\theta^{(9)}_{i}=
       \theta^{(11)}_{i}=\theta^{(13)}_{i}$ and $\theta^{(4)}_{i}=\theta^{(6)}_{i}=
       \theta^{(8)}_{i}=\theta^{(10)}_{i}=\theta^{(12)}_{i}=\theta^{(14)}_{i}.$ Using this,     relation $(\ref{112})$ becomes
       \begin{equation*}
        \frac{1}{h^{2}}(v_{i+1}-2v_{i}+v_{i-1})-v_{2x,i}=\frac{h^{2}}{12}v_{4x,i}-\frac{h^{4}}{720}\left[v_{6x}(\theta^{(3)}_{i})
        +v_{6x}(\theta^{(4)}_{i})\right].
       \end{equation*}
       This completes the proof of the first item of Lemma $\ref{l2}.$ Now, let prove the second item of Lemma $\ref{l2}.$\\

         Plugging equations $(\ref{76})$ and $(\ref{78}),$ $(\ref{77})$ and $(\ref{79}),$ $(\ref{78})$ and $(\ref{79}),$ respectively, it is
         not hard to see that
       \begin{equation}\label{100}
        v_{i+2}-2v_{i+1}+v_{i}=h^{2}v_{2x,i+1}+\frac{h^{4}}{12}v_{4x,i+1}+\frac{h^{6}}{720}\left(v_{6x}(\theta^{(1)}_{i})+(\theta^{(3)}_{i})\right);
       \end{equation}
       \begin{equation}\label{102}
        v_{i}-2v_{i-1}+v_{i-2}=h^{2}v_{2x,i-1}+\frac{h^{4}}{12}v_{4x,i-1}+\frac{h^{6}}{720}\left(v_{6x}(\theta^{(2)}_{i})+v_{6x}(\theta^{(3)}_{i})\right);
       \end{equation}
       and
       \begin{equation*}
        4v_{i}=2(v_{i+1}+v_{i-1})-2h(v_{x,i+1}-v_{x,i-1})+h^{2}(v_{2x,i+1}+v_{2x,i-1}) -\frac{h^{3}}{3}(v_{3x,i+1}-v_{3x,i-1})+
       \end{equation*}
       \begin{equation}\label{101}
        \frac{h^{4}}{12}(v_{4x,i+1}+v_{4x,i-1})-\frac{h^{5}}{60}(v_{5x,i+1}-v_{5x,i-1})+\frac{h^{6}}{720}\left(v_{6x}(\theta^{(3)}_{i})
        +v_{6x}(\theta^{(4)}_{i})\right).
       \end{equation}
       A combination of equations $(\ref{100})$-$(\ref{101})$ yields
       \begin{equation*}
        v_{i+2}-4v_{i+1}+6v_{i}-4v_{i-1}+v_{i-2}=-2h(v_{x,i+1}-v_{x,i-1})+2h^{2}(v_{2x,i+1}+v_{2x,i-1})-\frac{h^{3}}{3}(v_{3x,i+1}-v_{3x,i-1})+
       \end{equation*}
       \begin{equation}\label{103}
        \frac{h^{4}}{6}(v_{4x,i+1}+v_{4x,i-1})-\frac{h^{5}}{60}(v_{5x,i+1}-v_{5x,i-1})+\frac{h^{6}}{720}\left[v_{6x}(\theta^{(1)}_{i})
        +v_{6x}(\theta^{(2)}_{i})+3\left(v_{6x}(\theta^{(3)}_{i})+v_{6x}(\theta^{(4)}_{i})\right)\right].
       \end{equation}
       Substituting $(\ref{114})$-$(\ref{117})$ into $(\ref{103})$,simple computations result in
       \begin{equation*}
        v_{i+2}-4v_{i+1}+6v_{i}-4v_{i-1}+v_{i-2}=-4h^{2}v_{2x,i}-\frac{2h^{4}}{3}v_{4x,i}+4h^{2}v_{2x,i}+2h^{4}v_{4x,i}-\frac{2h^{4}}{3}v_{4x,i}+
        h^{6}\left\{\frac{1}{720}\left[v_{6x}(\theta^{(1)}_{i})\right.\right.
       \end{equation*}
       \begin{equation*}
        \left.+v_{6x}(\theta^{(2)}_{i})+3\left(v_{6x}(\theta^{(3)}_{i})+v_{6x}(\theta^{(4)}_{i})\right)\right]-\frac{1}{60}
        \left(v_{6x}(\theta^{(5)}_{i})+v_{6x}(\theta^{(6)}_{i})\right)+\frac{1}{12}\left(v_{6x}(\theta^{(7)}_{i})+v_{6x}(\theta^{(8)}_{i})\right)-
       \end{equation*}
       \begin{equation*}
        \left.\frac{1}{18}\left(v_{6x}(\theta^{(9)}_{i})+v_{6x}(\theta^{(10)}_{i})\right)+\frac{1}{12}\left(v_{6x}(\theta^{(11)}_{i})+
        v_{6x}(\theta^{(12)}_{i})\right)-\frac{1}{60}\left(v_{6x}(\theta^{(13)}_{i})+v_{6x}(\theta^{(14)}_{i})\right)\right\},
       \end{equation*}
       which is equivalent to
       \begin{equation*}
        v_{i+2}-4v_{i+1}+6v_{i}-4v_{i-1}+v_{i-2}=h^{4}v_{4x,i}+h^{6}\left\{\frac{1}{720}\left[v_{6x}(\theta^{(1)}_{i})+v_{6x}(\theta^{(2)}_{i})+
        3\left(v_{6x}(\theta^{(3)}_{i})+v_{6x}(\theta^{(4)}_{i})\right)\right]-\right.
       \end{equation*}
       \begin{equation*}
        \frac{1}{60}\left(v_{6x}(\theta^{(5)}_{i})+v_{6x}(\theta^{(6)}_{i})\right)+\frac{1}{12}\left(v_{6x}(\theta^{(7)}_{i})
        +v_{6x}(\theta^{(8)}_{i})\right)-\frac{1}{18}\left(v_{6x}(\theta^{(9)}_{i})+v_{6x}(\theta^{(10)}_{i})\right)
       \end{equation*}
       \begin{equation}\label{119}
        \left.+\frac{1}{12}\left(v_{6x}(\theta^{(11)}_{i})+v_{6x}(\theta^{(12)}_{i})\right)-\frac{1}{60}\left(v_{6x}(\theta^{(13)}_{i})
        +v_{6x}(\theta^{(14)}_{i})\right)\right\}.
       \end{equation}
       Assuming that $\theta^{(3)}_{i}=\theta^{(5)}_{i}=\theta^{(7)}_{i}=\theta^{(9)}_{i}=\theta^{(11)}_{i}=\theta^{(13)}_{i}$ and $\theta^{(4)}_{i}=\theta^{(6)}_{i}=\theta^{(8)}_{i}=\theta^{(10)}_{i}=\theta^{(12)}_{i}=\theta^{(14)}_{i},$ equation
       $(\ref{119})$ becomes
       \begin{equation*}
        v_{i+2}-4v_{i+1}+6v_{i}-4v_{i-1}+v_{i-2}=h^{4}v_{4x,i}+h^{6}\left\{\frac{1}{720}\left[v_{6x}(\theta^{(1)}_{i})+v_{6x}(\theta^{(2)}_{i})+
        \right]+\frac{241}{3220}\left[v_{6x}(\theta^{(3)}_{i})+v_{6x}(\theta^{(4)}_{i})\right]\right\},
       \end{equation*}
        which is equivalent to
        \begin{equation*}
        \frac{1}{h^{4}}(v_{i+2}-4v_{i+1}+6v_{i}-4v_{i-1}+v_{i-2})-v_{4x,i}=h^{2}\left\{\frac{1}{720}\left[v_{6x}(\theta^{(1)}_{i})
        +v_{6x}(\theta^{(2)}_{i})+\right]+\frac{241}{3220}\left[v_{6x}(\theta^{(3)}_{i})+v_{6x}(\theta^{(4)}_{i})\right]\right\}.
       \end{equation*}
       This ends the proof of Lemma $\ref{l2}.$
       \end{proof}

       \begin{lemma}\label{l3}
       The term $\rho_{ij}^{n}$ given by equation $(\ref{20})$ can be bounded as
       \begin{equation}\label{120}
        |\rho_{ij}^{n}|\leq \widehat{C}_{1}\left[1+\widehat{C}_{2}h^{2}+\widehat{C}_{3}h^{4}\right],
       \end{equation}
       where $\widehat{C}_{l},$ $l=1,2,3,$ are positive constant independent of the time step $k$ and the mesh size $h.$
       \end{lemma}

       \begin{proof} (of Lemma $\ref{l3}$).
       It comes from relation $(\ref{20})$ that
       \begin{equation*}
        \rho_{ij}^{n}=\frac{1}{8}\left[a^{2}\delta_{y}^{2}(\delta_{y}^{2}\overline{u}^{n}_{ij})+a\delta_{y}^{2}f(\overline{u}_{ij}^{n})
        +a\delta_{y}^{2}\overline{u}^{n}_{ij}f^{'}(\overline{u}_{ij}^{n})+f(\overline{u}_{ij}^{n})f^{'}(\overline{u}_{ij}^{n})\right].
       \end{equation*}
        From the definition of the operator $"\delta_{y}^{2}",$ this is equivalent to
        \begin{equation*}
        \rho_{ij}^{n}=\frac{1}{8}\left[\frac{a^{2}}{h^{4}}\left(\overline{u}_{i,j+2}^{n}-4\overline{u}_{i,j+1}^{n}
        +6\overline{u}_{ij}^{n}-4\overline{u}_{i,j-1}^{n}+\overline{u}_{i,j-2}\right)+\frac{a}{h^{2}}
        \left[f(\overline{u}_{i,j+1}^{n})-2f(\overline{u}_{ij}^{n})+f(\overline{u}_{i,j-1}^{n})+\right.\right.
       \end{equation*}
       \begin{equation*}
        \left.\left.\left(\overline{u}^{n}_{i,j+1}-2\overline{u}^{n}_{ij}+\overline{u}^{n}_{i,j-1}\right)f^{'}(\overline{u}_{ij}^{n})
        \right]+f(\overline{u}_{ij}^{n})f^{'}(\overline{u}_{ij}^{n})\right].
       \end{equation*}
       Combining this together with Lemma $\ref{l2},$ it is easy to see that
       \begin{equation*}
        \rho_{ij}^{n}=\frac{1}{8}\left\{a^{2}\left[\overline{u}_{4y,ij}^{n}+h^{2}\left(\frac{1}{720}\left[\overline{u}^{n}_{6y}(x_{i},\theta^{(1)}_{j})
        +\overline{u}^{n}_{6y}(x_{i},\theta^{(2)}_{j})\right]+\frac{241}{3220}\left[\overline{u}^{n}_{6y}(x_{i},\theta^{(3)}_{j})+
        \overline{u}^{n}_{6y}(x_{i},\theta^{(4)}_{j})\right]\right)\right]+\right.
       \end{equation*}
       \begin{equation*}
        \left.a\left[fo\overline{u})^{n}_{2y,ij}+\frac{h^{2}}{12}fo\overline{u})^{n}_{4y,ij}-\frac{h^{4}}{720}
        \left(fo\overline{u})^{n}_{6y}(x_{i},\theta^{(3)}_{j})+fo\overline{u})^{n}_{6y}(x_{i},\theta^{(4)}_{j})
        \right)+\left(\overline{u}^{n}_{2y,ij}+\frac{h^{2}}{12}\overline{u}^{n}_{4y,ij}-\right.\right.\right.
       \end{equation*}
        \begin{equation*}
        \left.\left.\left.\frac{h^{4}}{720}\left(\overline{u}^{n}_{6y}(x_{i},\theta^{(3)}_{j})+\overline{u}^{n}_{6y}(x_{i},\theta^{(4)}_{j})
        \right)\right)f^{'}(\overline{u}_{ij}^{n})\right]+f(\overline{u}_{ij}^{n})f^{'}(\overline{u}_{ij}^{n})\right\}.
        \end{equation*}
        On the other hand, $\overline{u}(x,\cdot,t)|_{[y_{j},y_{j+1}]}$, $fo\overline{u}(x,\cdot,t)|_{[y_{j},y_{j+1}]}\in
        \mathcal{C}^{6}([y_{j},y_{j+1}])$, for every $x\in(0,1)$, $t\in(0,T)$ and $j=0,1,2,...,M-1;$ $\||\overline{u}|\|_{L^{\infty}
        (0,T;L^{2}(\Omega))}\leq\widetilde{C}$ (according to estimate $(\ref{43})$) and $f^{'}$ (the derivative of $f$) is continuous.
        Taking the absolute value of $\rho_{ij}^{n},$ there exist positive constants $\widehat{C}_{l},$ $l=1,2,3,$ independent of the time step $k$
        and the mesh grid $h$ so that
         \begin{equation*}
        |\rho_{ij}^{n}|\leq \widehat{C}_{1}\left[1+\widehat{C}_{2}h^{2}+\widehat{C}_{3}h^{4}\right].
       \end{equation*}
       This completes the proof of Lemma $\ref{l3}.$
       \end{proof}

        Armed with the results provided by Lemmas $\ref{l1}$, $\ref{l2}$ and $\ref{l3},$ we are ready to prove Theorem $\ref{t2}.$

      \begin{proof} (of Theorem $\ref{t2}$)
        We recall that the error term provided by the scheme $(\ref{40})$-$(\ref{24})$ is denoted by $e_{ij}^{n}=u_{ij}^{n}-\overline{u}_{ij}^{n},$
        where $\overline{u}$ satisfies equations $(\ref{20a}),$ $(\ref{31})$ and $(\ref{33})$ and $u$ is given by relations
        $(\ref{40})$-$(\ref{24})$. So, it comes from equation $(\ref{46})$ that
       \begin{equation*}
        e_{ij}^{*}=e^{n}_{ij}+\frac{k}{2}\left(\frac{a}{h^{2}}(e^{n}_{i,j+1}-2e^{n}_{ij}+e^{n}_{i,j-1})+f(u_{ij}^{n})-f(\overline{u}_{ij}^{n})\right)
        +\frac{1}{2}k^{2}(3\rho_{ij}^{n}+\rho_{ij}^{\overline{*}})+O(k^{3}+kh^{2}),
       \end{equation*}
       which is equivalent to
       \begin{equation*}
        e_{ij}^{*}=e^{n}_{ij}+\frac{k}{2}\left(\frac{a}{h^{2}}(e^{n}_{i,j+1}-2e^{n}_{ij}+e^{n}_{i,j-1})+f(u_{ij}^{n})-f(\overline{u}_{ij}^{n})\right)
        +\frac{1}{2}k^{2}(3\rho_{ij}^{n}+\rho_{ij}^{\overline{*}})+C_{r}(k^{3}+kh^{2}),
       \end{equation*}
        where $C_{r}$ is a parameter that does not depend neither on the time step $k$ nor the grid spacing $h$ and $\rho_{ij}^{\alpha}$ is
        defined by $(\ref{20})$. Taking the square, it is not hard to see that
        \begin{equation*}
        (e_{ij}^{*})^{2}=(e^{n}_{ij})^{2}+k\left\{\frac{a}{h^{2}}\left(e^{n}_{i,j+1}-2e^{n}_{ij}+e^{n}_{i,j-1}\right)e^{n}_{ij}+\left(f(u_{ij}^{n})-
        f(\overline{u}_{ij}^{n})\right)e^{n}_{ij}\right\}+k^{2}(3\rho_{ij}^{n}+\rho_{ij}^{\overline{*}})e^{n}_{ij}+2C_{r}(k^{3}+kh^{2})e^{n}_{ij}
        \end{equation*}
       \begin{equation*}
       +\frac{k^{2}}{4}\left(\frac{a}{h^{2}}(e^{n}_{i,j+1}-2e^{n}_{ij}+e^{n}_{i,j-1})+f(u_{ij}^{n})-f(\overline{u}_{ij}^{n})\right)^{2}+
       \frac{k^{3}}{2}\left(\frac{a}{h^{2}}(e^{n}_{i,j+1}-2e^{n}_{ij}+e^{n}_{i,j-1})+f(u_{ij}^{n})-f(\overline{u}_{ij}^{n})\right)(3\rho_{ij}^{n}
       \end{equation*}
       \begin{equation*}
        +\rho_{ij}^{\overline{*}})+C_{r}\left(\frac{a}{h^{2}}(e^{n}_{i,j+1}-2e^{n}_{ij}+e^{n}_{i,j-1})+f(u_{ij}^{n})-f(\overline{u}_{ij}^{n})\right)
        (k^{4}+k^{2}h^{2})+\frac{k^{4}}{4}(3\rho_{ij}^{n}+\rho_{ij}^{\overline{*}})^{2}+C_{r}^{2}(k^{3}+kh^{2})^{2}+
       \end{equation*}
       \begin{equation}\label{67}
        C_{r}(k^{5}+k^{3}h^{2})(3\rho_{ij}^{n}+\rho_{ij}^{\overline{*}}).
       \end{equation}
         Applying the inequalities: $2ab\leq a^{2}+b^{2},$ $(a\pm b)^{2}\leq 2(a^{2}+b^{2})$ and $(a\pm b\pm c)^{2}\leq3(a^{2}+b^{2}+c^{2}),$
         for every $a,b,c\in\mathbb{R},$ together with the time step restriction $(\ref{44})$ (that is, $2ak\leq h^{2}$), relation $(\ref{67})$ becomes
        \begin{equation*}
        (e_{ij}^{*})^{2}\leq(e^{n}_{ij})^{2}+k\left\{\frac{a}{h^{2}}\left(e^{n}_{i,j+1}-2e^{n}_{ij}+e^{n}_{i,j-1}\right)e^{n}_{ij}
        +\left(f(u_{ij}^{n})-f(\overline{u}_{ij}^{n})\right)e^{n}_{ij}\right\}+k^{2}|(3\rho_{ij}^{n}+\rho_{ij}^{\overline{*}})e^{n}_{ij}|+2C_{r}(k^{3}
        +kh^{2})|e^{n}_{ij}|
        \end{equation*}
       \begin{equation*}
       +\frac{k^{2}}{4}\left\{\left[\frac{a}{h^{2}}(e^{n}_{i,j+1}-2e^{n}_{ij}+e^{n}_{i,j-1})\right]^{2}+\left[f(u_{ij}^{n})-f(\overline{u}_{ij}^{n})
       \right]^{2}+2\frac{a}{h^{2}}(e^{n}_{i,j+1}-2e^{n}_{ij}+e^{n}_{i,j-1})\left[f(u_{ij}^{n})-f(\overline{u}_{ij}^{n})\right]\right\}+
       \end{equation*}
       \begin{equation*}
        \frac{k^{2}}{4}\left(|e^{n}_{i,j+1}-2e^{n}_{ij}+e^{n}_{i,j-1}|+\frac{h^{2}}{a}|f(u_{ij}^{n})-f(\overline{u}_{ij}^{n})|\right)
        |3\rho_{ij}^{n}+\rho_{ij}^{\overline{*}}|+\frac{C_{r}}{2}\left(|e^{n}_{i,j+1}-2e^{n}_{ij}+e^{n}_{i,j-1}|+\frac{h^{2}}{a}|f(u_{ij}^{n})-\right.
       \end{equation*}
       \begin{equation*}
        \left.f(\overline{u}_{ij}^{n})|\right)(k^{3}+kh^{2})+\frac{k^{4}}{2}(3\rho_{ij}^{n}+\rho_{ij}^{\overline{*}})^{2}+2C_{r}^{2}(k^{3}+kh^{2})^{2}.
       \end{equation*}
       which implies
       \begin{equation*}
        (e_{ij}^{*})^{2}\leq(e^{n}_{ij})^{2}+k\left\{\frac{a}{h^{2}}\left(e^{n}_{i,j+1}-2e^{n}_{ij}+e^{n}_{i,j-1}\right)e^{n}_{ij}
        +\left(f(u_{ij}^{n})-f(\overline{u}_{ij}^{n})\right)e^{n}_{ij}\right\}+k^{2}|(3\rho_{ij}^{n}+\rho_{ij}^{\overline{*}})e^{n}_{ij}|+2C_{r}(k^{3}
        +kh^{2})|e^{n}_{ij}|
        \end{equation*}
       \begin{equation*}
       +\frac{k^{2}}{4}\left\{\frac{2a^{2}}{h^{4}}\left[(e^{n}_{i,j+1}-e^{n}_{ij})^{2}+(e^{n}_{ij}-e^{n}_{i,j-1})^{2}\right]+\left[f(u_{ij}^{n})-
       f(\overline{u}_{ij}^{n})\right]^{2}+\frac{a}{h^{2}}\left[(e^{n}_{i,j+1}-2e^{n}_{ij}+e^{n}_{i,j-1})^{2}+\right.\right.
       \end{equation*}
       \begin{equation*}
        \left.\left.(f(u_{ij}^{n})-f(\overline{u}_{ij}^{n}))^{2}\right]\right\}+\frac{k^{2}}{4}\left(|e^{n}_{i,j+1}-2e^{n}_{ij}+e^{n}_{i,j-1}|
        +\frac{h^{2}}{a}|f(u_{ij}^{n})-f(\overline{u}_{ij}^{n})|\right)|3\rho_{ij}^{n}+\rho_{ij}^{\overline{*}}|+\frac{C_{r}}{2}\left(|e^{n}_{i,j+1}
        -2e^{n}_{ij}\right.
       \end{equation*}
       \begin{equation*}
        \left.+e^{n}_{i,j-1}|+\frac{h^{2}}{a}|f(u_{ij}^{n})-f(\overline{u}_{ij}^{n})|\right)(k^{3}+kh^{2})+\frac{k^{4}}{2}(3\rho_{ij}^{n}
        +\rho_{ij}^{\overline{*}})^{2}+2C_{r}^{2}(k^{3}+kh^{2})^{2}.
       \end{equation*}
       \begin{equation*}
        \leq(e^{n}_{ij})^{2}+k\left\{\frac{a}{h^{2}}\left(e^{n}_{i,j+1}-2e^{n}_{ij}+e^{n}_{i,j-1}\right)e^{n}_{ij}+\left(f(u_{ij}^{n})-
        f(\overline{u}_{ij}^{n})\right)e^{n}_{ij}\right\}+\frac{1}{2}\left[k^{3}(3\rho_{ij}^{n}+\rho_{ij}^{\overline{*}})^{2}+
        8C_{r}^{2}(k^{5}+kh^{4})\right]
        \end{equation*}
       \begin{equation*}
       +k(e^{n}_{ij})^{2}+\frac{k^{2}}{4}\left\{\frac{2a^{2}}{h^{2}}\left[(\delta_{y}e^{n}_{i,j+\frac{1}{2}})^{2}+(\delta_{y}e^{n}_{i,j
       -\frac{1}{2}})^{2}\right]+\left[f(u_{ij}^{n})-f(\overline{u}_{ij}^{n})\right]^{2}+\frac{a}{h^{2}}\left[3\left((e^{n}_{i,j+1})^{2}+4(e^{n}_{ij})^{2}
       +(e^{n}_{i,j-1})^{2}\right)+\right.\right.
       \end{equation*}
       \begin{equation*}
        \left.\left.(f(u_{ij}^{n})-f(\overline{u}_{ij}^{n}))^{2}\right]\right\}+\frac{k^{3}}{4}(3\rho_{ij}^{n}+\rho_{ij}^{\overline{*}})^{2}
        +\frac{3k}{8}\left[(e^{n}_{i,j+1})^{2}+4(e^{n}_{ij})^{2}+(e^{n}_{i,j-1})^{2}\right]+\frac{kh^{4}}{8a^{2}}(f(u_{ij}^{n})
        -f(\overline{u}_{ij}^{n}))^{2}
       \end{equation*}
       \begin{equation*}
        +C_{r}^{2}(k^{\frac{3}{2}}+k^\frac{1}{2}h^{2})^{2}+\frac{3k}{8}\left((e^{n}_{i,j+1})^{2}+4(e^{n}_{ij})^{2}+(e^{n}_{i,j-1})^{2}\right)+        \frac{kh^{4}}{8a^{2}}(f(u_{ij}^{n})-f(\overline{u}_{ij}^{n}))^{2}+\frac{k^{4}}{2}(3\rho_{ij}^{n}+\rho_{ij}^{\overline{*}})^{2}
       \end{equation*}
       \begin{equation}\label{68}
        +2C_{r}^{2}(k^{3}+kh^{2})^{2}.
       \end{equation}
        From estimates $(\ref{53})$-$(\ref{54})$, we have that
        \begin{equation*}
        |f(u_{ij}^{n})-f(\overline{u}_{ij}^{n})|\leq C|e_{ij}^{n}|;\text{\,\,\,}\left(f(u_{ij}^{n})-f(\overline{u}_{ij}^{n})\right)e_{ij}^{n}\leq C(e_{ij}^{n})^{2}\text{\,\,\,and\,\,\,}\left(f(u_{ij}^{n})-f(\overline{u}_{ij}^{n})\right)^{2}\leq C^{2}(e_{ij}^{n})^{2}.
        \end{equation*}
        This fact, together with estimate $(\ref{68})$ results in
       \begin{equation*}
        (e_{ij}^{*})^{2}\leq (e^{n}_{ij})^{2}+k\left\{\frac{a}{h^{2}}\left(e^{n}_{i,j+1}-2e^{n}_{ij}+e^{n}_{i,j-1}\right)e^{n}_{ij}
        +C(e_{ij}^{n})^{2}\right\}+\frac{1}{2}\left[k^{3}(3\rho_{ij}^{n}+\rho_{ij}^{\overline{*}})^{2}+8C_{r}^{2}(k^{5}+kh^{4})\right]
        \end{equation*}
       \begin{equation*}
       +k(e^{n}_{ij})^{2}+\frac{k^{2}}{4}\left\{\frac{2a^{2}}{h^{2}}\left[(\delta_{y}e^{n}_{i,j+\frac{1}{2}})^{2}+(\delta_{y}e^{n}_{i,j
       -\frac{1}{2}})^{2}\right]+C^{2}(e_{ij}^{n})^{2}+\frac{a}{h^{2}}\left[3\left((e^{n}_{i,j+1})^{2}+4(e^{n}_{ij})^{2}+(e^{n}_{i,j-1})^{2}\right)
       +\right.\right.
       \end{equation*}
       \begin{equation*}
        \left.\left.C^{2}(e_{ij}^{n})^{2}\right]\right\}+\frac{k^{3}}{4}(3\rho_{ij}^{n}+\rho_{ij}^{\overline{*}})^{2}
        +\frac{3k}{4}\left[(e^{n}_{i,j+1})^{2}+4(e^{n}_{ij})^{2}+(e^{n}_{i,j-1})^{2}\right]+C^{2}\frac{kh^{4}}{4a^{2}}(e_{ij}^{n})^{2}+
        2C_{r}^{2}(k^{3}+kh^{4})
       \end{equation*}
       \begin{equation*}
        +\frac{k^{4}}{2}(3\rho_{ij}^{n}+\rho_{ij}^{\overline{*}})^{2}+2C_{r}^{2}(k^{3}+kh^{2})^{2}.
       \end{equation*}
       Utilizing the time step restriction $(\ref{44})$, $\frac{2ak}{h^{2}}\leq1,$ this implies
       \begin{equation*}
        (e_{ij}^{*})^{2}\leq(e^{n}_{ij})^{2}+\frac{ak}{h^{2}}\left(e^{n}_{i,j+1}-2e^{n}_{ij}+e^{n}_{i,j-1}\right)e^{n}_{ij}+\frac{ak}{4}
        \left[(\delta_{x}e^{n}_{i,j+\frac{1}{2}})^{2}+(\delta_{y}e^{n}_{i,j-\frac{1}{2}})^{2}\right]+k\left[1+C+\frac{C^{2}}{8}+\right.
        \end{equation*}
       \begin{equation*}
       \left.\frac{C^{2}k}{4}+\frac{C^{2}h^{4}}{4a^{2}}\right](e_{ij}^{n})^{2}+\frac{9k}{8}\left[(e_{i,j+1}^{n})^{2}+4(e_{ij}^{n})^{2}
       +(e_{i,j-1}^{n})^{2}\right]+2kC_{r}^{2}(k^{2}+h^{4})+4k^{2}C_{r}^{2}(k^{4}+h^{4})+
       \end{equation*}
       \begin{equation*}
        8kC_{r}^{2}(k^{4}+h^{4})+\frac{k^{3}}{4}(3+2k)(3\rho_{ij}^{n}+\rho_{ij}^{\overline{*}})^{2}.
       \end{equation*}
       Summing this up from $i,j=1,2,...M-1,$ provides

       \begin{equation*}
        \underset{i,j=1}{\overset{M-1}\sum}(e_{ij}^{*})^{2}\leq\underset{i,j=1}{\overset{M-1}\sum}(e^{n}_{ij})^{2}
        +k\left[\frac{a}{h^{2}}\underset{i,j=1}{\overset{M-1}\sum}\left(e^{n}_{i,j+1}-2e^{n}_{ij}+e^{n}_{i,j-1}\right)e^{n}_{ij}\right]
        +\frac{ak}{4}\underset{i,j=1}{\overset{M-1}\sum}\left[(\delta_{y}e^{n}_{i,j+\frac{1}{2}})^{2}+(\delta_{y}e^{n}_{i,j-\frac{1}{2}})^{2}\right]+
        \end{equation*}
       \begin{equation*}
       k\left[1+C+\frac{C^{2}}{8}+\frac{C^{2}k}{4}+\frac{C^{2}h^{4}}{4a^{2}}\right]\underset{i,j=1}{\overset{M-1}\sum}(e_{ij}^{n})^{2}
       +\frac{9k}{8}\underset{i,j=1}{\overset{M-1}\sum}\left[(e_{i,j+1}^{n})^{2}+4(e_{ij}^{n})^{2}+(e_{i,j-1}^{n})^{2}\right]
       \end{equation*}
       \begin{equation}\label{56a}
        +\frac{k^{3}}{4}(3+2k)\underset{i,j=1}{\overset{M-1}\sum}[9(\rho_{ij}^{n})^{2}+(\rho_{ij}^{\overline{*}})^{2}]+2C_{r}^{2}k
        \underset{i,j=1}{\overset{M-1}\sum}\left[k^{2}+4k^{4}+5h^{4}+2k(k^{4}+h^{4})\right].
       \end{equation}
        Combining the boundary condition $(\ref{24})$, $e^{n}_{Mj}=e^{n}_{0j}=0,$ for all $j=0,1,...,M,$ Lemmas $\ref{l1}$ and $\ref{l3},$
        and multiplying both sides of inequality $(\ref{56a})$ by $h^{2},$ straightforward computations yield
       \begin{equation*}
        h^{2}\underset{i,j=1}{\overset{M-1}\sum}(e_{ij}^{*})^{2}\leq h^{2}\underset{i,j=1}{\overset{M-1}\sum}(e^{n}_{ij})^{2}
        -ak\|\delta_{y}e^{n}\|_{L^{2}(\Omega)}^{2}+\frac{ak}{2}\|\delta_{y}e^{n}\|_{L^{2}(\Omega)}^{2}+
        k\left[\frac{31}{4}+C+\frac{C^{2}}{8}+\frac{C^{2}k}{4}+\frac{C^{2}h^{4}}{4a^{2}}\right]
        \end{equation*}
       \begin{equation*}
       h^{2}\underset{i,j=1}{\overset{M-1}\sum}(e_{ij}^{n})^{2}+\frac{k^{3}h^{2}}{4}(3+2k)(M-1)^{2}\left[9\widehat{C}_{1}^{2}
       (1+\widehat{C}_{2}h^{2}+\widehat{C}_{3}h^{4})^{2}+\widehat{C}_{1}^{2}(1+\widehat{C}_{2}h^{2}+\widehat{C}_{3}h^{4})^{2}\right]
       \end{equation*}
       \begin{equation*}
        +2C_{r}^{2}kh^{2}(M-1)^{2}\left[k^{2}+4k^{4}+5h^{4}+2k(k^{4}+h^{4})\right].
       \end{equation*}
        Since $h=\frac{1}{M},$ $k\leq 1+k^{2}$ and $h^{2}\leq 1+h^{4},$ this becomes
        \begin{equation*}
        h^{2}\underset{j,i=1}{\overset{M-1}\sum}(e_{ij}^{*})^{2}\leq h^{2}\underset{j,i=1}{\overset{M-1}\sum}(e^{n}_{ij})^{2}
        -\frac{ak}{2}\|\delta_{y}e^{n}\|_{L^{2}(\Omega)}^{2}+k\left[\frac{31}{2}+C+\frac{C^{2}}{8}+\frac{C^{2}k}{4}+\frac{C^{2}h^{4}}{4a^{2}}\right]
        h^{2}\underset{j,i=1}{\overset{M-1}\sum}(e_{ij}^{n})^{2}+
        \end{equation*}
       \begin{equation*}
       \frac{5\widehat{C}_{1}^{2}k^{3}}{2}(5+2k^{2})(1+\widehat{C}_{2}+\widehat{C}_{2}h^{4}+\widehat{C}_{3}h^{4})^{2}+2C_{r}^{2}k
       \left[k^{2}+4k^{4}+5h^{4}+2k(k^{4}+h^{4})\right].
       \end{equation*}
       which implies
       \begin{equation*}
        \|e^{*}\|_{L^{2}(\Omega)}^{2}\leq\|e^{n}\|_{L^{2}(\Omega)}^{2}+\widehat{C}_{4}\left\{k\left[1+k+h^{4}\right]
        \|e^{n}\|_{L^{2}(\Omega)}^{2}+k^{3}\left[1+k^{2}+h^{4}+h^{6}+h^{8}+\right.\right.
       \end{equation*}
       \begin{equation}\label{121}
        \left.\left.k^{2}h^{2}+k^{2}h^{4}+k^{2}h^{6}+k^{2}h^{8}\right]+k\left(1+k\right)(k^{4}+h^{4})\right\},
       \end{equation}
       where we absorbed all the constants into a constant $\widehat{C}_{4}.$\\

        Similarly, one shows that
         \begin{equation*}
        \|e^{**}\|_{L^{2}(\Omega)}^{2}\leq\|e^{*}\|_{L^{2}(\Omega)}^{2}+\widehat{C}_{5}\left\{k\left[1+k+h^{4}\right]
        \|e^{*}\|_{L^{2}(\Omega)}^{2}+k^{3}\left[1+k^{2}+h^{4}+h^{6}+h^{8}+\right.\right.
       \end{equation*}
       \begin{equation}\label{122}
        \left.\left.k^{2}h^{2}+k^{2}h^{4}+k^{2}h^{6}+k^{2}h^{8}\right]+k\left(1+k\right)(k^{4}+h^{4})\right\},
       \end{equation}
       where all the constants have been absorbed into a constant $\widehat{C}_{5},$ and
        \begin{equation*}
        \|e^{n+1}\|_{L^{2}(\Omega)}^{2}\leq\|e^{**}\|_{L^{2}(\Omega)}^{2}+\widehat{C}_{6}\left\{k\left[1+k+h^{4}\right]
        \|e^{**}\|_{L^{2}(\Omega)}^{2}+k^{3}\left[1+k^{2}+h^{4}+h^{6}+h^{8}+\right.\right.
       \end{equation*}
       \begin{equation}\label{123}
        \left.\left.k^{2}h^{2}+k^{2}h^{4}+k^{2}h^{6}+k^{2}h^{8}\right]+k\left(1+k\right)(k^{4}+h^{4})\right\},
       \end{equation}
       where all the constants have been absorbed into a constant $\widehat{C}_{6}.$\\

       Now, setting
       \begin{equation}\label{124}
        \varphi_{1}(k,h)=1+k+h^{2}+h^{4},
       \end{equation}
        and
       \begin{equation}\label{125}
        \varphi_{2}(k,h)=k^{3}\left[1+k^{2}+h^{4}+h^{6}+h^{8}+k^{2}h^{2}+k^{2}h^{4}+k^{2}h^{6}+k^{2}h^{8}\right]+k(1+k)(k^{4}+h^{4}),
       \end{equation}
        plugging estimates $(\ref{121})$-$(\ref{123}),$ straightforward calculations yield
       \begin{equation*}
        \|e^{n+1}\|_{L^{2}(\Omega)}^{2}\leq\|e^{n}\|_{L^{2}(\Omega)}^{2}+k\left\{\widehat{C}_{4}+\widehat{C}_{5}+\widehat{C}_{6}
        +k\left[\widehat{C}_{4}\widehat{C}_{5}+\widehat{C}_{6}(\widehat{C}_{4}+\widehat{C}_{5})+k\widehat{C}_{4}\widehat{C}_{5}\widehat{C}_{6}
        \varphi_{1}(k,h)\right]\varphi_{1}(k,h)\right\}
       \end{equation*}
       \begin{equation*}
        \varphi_{1}(k,h)\|e^{n}\|_{L^{2}(\Omega)}^{2}+\left[\widehat{C}_{4}+\widehat{C}_{5}+\widehat{C}_{6}+k\left[\widehat{C}_{4}\widehat{C}_{5}
        +\widehat{C}_{4}\widehat{C}_{6}+\widehat{C}_{5}\widehat{C}_{6}+k\widehat{C}_{4}\widehat{C}_{5}\widehat{C}_{6}\varphi_{1}(k,h)\right]
        \varphi_{1}(k,h)\right]
        \varphi_{2}(k,h).
       \end{equation*}
       Absorbing all the constants into a constant $\widehat{C}_{7},$ this yields
       \begin{equation*}
        \|e^{n+1}\|_{L^{2}(\Omega)}^{2}\leq\|e^{n}\|_{L^{2}(\Omega)}^{2}+\widehat{C}_{7}\left\{k\left[1+k\left(1+k\varphi_{1}(k,h)\right)
        \varphi_{1}(k,h)\right]\varphi_{1}(k,h)\|e^{n}\|_{L^{2}(\Omega)}^{2}
        \right.
       \end{equation*}
       \begin{equation*}
        \left.+\left[1+k\left[1+k\varphi_{1}(k,h)\right]\varphi_{1}(k,h)\right]\varphi_{2}(k,h)\right\}.
       \end{equation*}
       Summing this up from $n=0,1,2,..,p-1,$ for any nonnegative integer $p$ such that $1\leq p\leq N,$ we obtain
       \begin{equation*}
        \|e^{p}\|_{L^{2}(\Omega)}^{2}\leq\|e^{0}\|_{L^{2}(\Omega)}^{2}+\widehat{C}_{7}\left\{pk\left[1+k\left(1+k\varphi_{1}(k,h)\right)
        \varphi_{1}(k,h)\right]\varphi_{1}(k,h)\underset{n=0}{\overset{p-1}\sum}\|e^{n}\|_{L^{2}(\Omega)}^{2}\right.
       \end{equation*}
       \begin{equation}\label{126}
        \left.+p\left[1+k\left[1+k\varphi_{1}(k,h)\right]\varphi_{1}(k,h)\right]\varphi_{2}(k,h)\right\}.
       \end{equation}
       It comes from the initial condition given in $(\ref{24})$, that $e^{0}_{ij}=0,$ for $0\leq i,j\leq M.$ Applying the Gronwall Lemma, estimate $(\ref{126})$ provides
       \begin{equation}\label{127}
        \|e^{p}\|_{L^{2}(\Omega)}^{2}\leq\widehat{C}_{7}\exp\left\{\widehat{C}_{7}pk\left[1+k\left(1+k\varphi_{1}(k,h)\right)
        \varphi_{1}(k,h)\right]\varphi_{1}(k,h)\right\}p\left[1+k\left[1+k\varphi_{1}(k,h)\right]\varphi_{1}(k,h)\right]\varphi_{2}(k,h).
       \end{equation}
       But $k=\frac{T}{N},$ so $\widehat{C}_{7}kp=\widehat{C}_{7}T\frac{p}{N}\leq\widehat{C}_{7}T$ (since $p\leq N$).
       This fact, together with inequality $(\ref{127})$ result in
       \begin{equation*}
        \|e^{p}\|_{L^{2}(\Omega)}^{2}\leq\widehat{C}_{7}T\exp\left\{\widehat{C}_{7}T\left[1+k\left(1+k\varphi_{1}(k,h)\right)
        \varphi_{1}(k,h)\right]\varphi_{1}(k,h)\right\}\left[1+k\left[1+k\varphi_{1}(k,h)\right]\varphi_{1}(k,h)\right]\varphi_{3}(k,h)^{2},
       \end{equation*}
       where $\varphi_{3}(k,h)^{2}=k^{-1}\varphi_{2}(k,h),$ $\varphi_{2}(k,h)$ is given by equation $(\ref{125}).$
       Taking the square root, it is easy to see that
       \begin{equation}\label{128}
        \|e^{p}\|_{L^{2}(\Omega)}\leq\sqrt{\widehat{C}_{7}T\left[1+k\left[1+k\varphi_{1}(k,h)\right]\varphi_{1}(k,h)\right]}
        \exp\left\{\frac{\widehat{C}_{7}T}{2}\left[1+k\left(1+k\varphi_{1}(k,h)\right)
        \varphi_{1}(k,h)\right]\varphi_{1}(k,h)\right\}\varphi_{3}(k,h).
       \end{equation}
       It comes from equality $\varphi_{3}(k,h)^{2}=k^{-1}\varphi_{2}(k,h),$ and equation $(\ref{125})$ that
              \begin{equation*}
        \varphi_{3}(k,h)^{2}=k^{2}\left[1+k^{2}+h^{4}+h^{6}+h^{8}+k^{2}h^{2}+k^{2}h^{4}+k^{2}h^{6}+k^{2}h^{8}\right]+(1+k)(k^{4}+h^{4})\leq
       \end{equation*}
       \begin{equation*}
        (k+kh^{2})^{2}(\widetilde{C}_{8}+\varphi_{4}(k,h)),
       \end{equation*}
       where $\widetilde{C}_{8}$ is a positive constant independent of $k$ and $h,$ and $\varphi_{4}(k,h)$ tends to zero when $k,h\rightarrow0.$
       Taking the maximum over $p$ of estimate $(\ref{128})$, for $0\leq p\leq N$, the proof of Theorem $\ref{t2}$ is completed thanks to equation $(\ref{66})$.
      \end{proof}

         \section{Numerical experiments and Convergence rate}\label{sec5}
         In this section we construct an exact solution to the initial-boundary value problem $(\ref{1})$-$(\ref{3})$ for a specific source term $f$.
         Furthermore, using Matlab we perform some numerical experiments in bidimensional case. In that case we obtain satisfactory results,
         so our algorithm performances are not worse for multidimensional problems. We consider two cases which are physical examples associated with the
         diffusive coefficient $a=1,$ together with the example introduced in \cite{30wcld}. We confirm the predicted convergence rate from the theory
         (see Section $\ref{sec2}$, Page $6$, last paragraph). This convergence rate is obtained by listing in Tables $1$-$6$ the errors between the computed solution and the exact one with different values of mesh size $h$ and time step $k,$ satisfying $k=\frac{1}{2}h^{2}.$ Finally, we look at the error estimates of our proposed method for the parameter $T=1.$\\

         Assuming that the exact solution to problem $(\ref{1})$-$(\ref{3})$ is of the form $\overline{u}(x,y,t)=\left[1+\exp(ct+dx+by)\right]^{-n},$ where $n$ is an integer. By simple calculations, it holds
         \begin{equation}\label{1n}
            \overline{u}_{t}(x,y,t)=-nc\exp(ct+dx+by)\left[1+\exp(ct+dx+by)\right]^{-n-1},
         \end{equation}
         \begin{equation*}
            \overline{u}_{x}(x,y,t)=-nd\exp(ct+dx+by)\left[1+\exp(ct+dx+by)\right]^{-n-1},
         \end{equation*}
         and
         \begin{equation}\label{2n}
          \overline{u}_{xx}(x,y,t)=-nd^{2}\exp(ct+dx+by)\left[1-n\exp(ct+dx+by)\right]\left[1+\exp(ct+dx+by)\right]^{-n-2}.
         \end{equation}
         In way similar
         \begin{equation}\label{3n}
          \overline{u}_{yy}(x,y,t)=-nb^{2}\exp(ct+dx+by)\left[1-n\exp(ct+dx+by)\right]\left[1+\exp(ct+dx+by)\right]^{-n-2}.
         \end{equation}
         Combining equations $(\ref{1n})$-$(\ref{3n}),$ it is not hard to see that
         \begin{equation*}
         \overline{u}_{t}-(\overline{u}_{xx}+\overline{u}_{yy})=-n\exp(ct+dx+by)\left(1+\exp(ct+dx+by)\right)^{-n-1}\left\{c-(d^{2}+b^{2})
         \left[1-n\exp(ct+dx+by)\right]\right.
         \end{equation*}
         \begin{equation*}
            \left.\left(1+\exp(ct+dx+by)\right)^{-1}\right\}
         \end{equation*}
         Setting $c=-(d^{2}+b^{2}),$ this becomes
         \begin{equation*}
         \overline{u}_{t}-(\overline{u}_{xx}+\overline{u}_{yy})=n(d^{2}+b^{2})\exp(ct+dx+by)\left(1+\exp(ct+dx+by)\right)^{-n-1}
         \left\{1+\left[1-n\exp(ct+dx+by)\right]\right.
         \end{equation*}
         \begin{equation*}
            \left.\left(1+\exp(ct+dx+by)\right)^{-1}\right\}=n(d^{2}+b^{2})\exp(ct+dx+by)\left[2+(1-n)\exp(ct+dx+by))\right]
         \end{equation*}
         \begin{equation}\label{4n}
            \left(1+\exp(ct+dx+by)\right)^{-n-2}.
         \end{equation}

          $\bullet$: \textbf{Case 1: $n=1$.} In this case, equation $(\ref{4n})$ becomes
         \begin{equation*}
         \overline{u}_{t}-(\overline{u}_{xx}+\overline{u}_{yy})=2(d^{2}+b^{2})\exp(ct+dx+by)\left(1+\exp(ct+dx+by)\right)^{-3}.
         \end{equation*}
         Now, taking $2(d^{2}+b^{2})=1,$ this gives $b^{2}=\frac{1}{2}-d^{2}.$ Since $b^{2}$ must be strictly greater than zero, this implies
         $d^{2}<\frac{1}{2}.$ For $d=\pm\frac{\sqrt{3}}{3},$ this implies $b=\pm\frac{\sqrt{6}}{6}$ and $c=-\frac{1}{2}.$ Letting $f(\overline{u})=(1-\overline{u})\overline{u}^{2},$ our exact solution is given by $\overline{u}(x,y,t)=\left[1+\exp\left(-\frac{1}{2}t+\frac{\sqrt{3}}{3}x+\frac{\sqrt{6}}{6}y\right)\right]^{-1},$ for
         $t\in[0,1]$ and $(x,y)\in[0,1]^{2}.$ The initial and boundary conditions are determined by this solution.\\

         $\bullet$: \textbf{Case 2:  $n=-1$.} It comes from equation $(\ref{4n})$ that
         \begin{equation*}
         \overline{u}_{t}-(\overline{u}_{xx}+\overline{u}_{yy})=-2(d^{2}+b^{2})\exp(ct+dx+by).
         \end{equation*}
          Since $\overline{u}=1+\exp(ct+dx+by),$ so $-\exp(ct+dx+by)=1-u.$ Taking $-2(d^{2}+b^{2})=-1,$ it holds $d=\pm\frac{\sqrt{3}}{3},$ $b=\pm\frac{\sqrt{6}}{6}$ and $c=-\frac{1}{2}.$ Setting $f(\overline{u})=1-\overline{u},$ we consider the exact solution defined as $\overline{u}(x,y,t)=1+\exp\left(-\frac{1}{2}t+\frac{\sqrt{3}}{3}x+\frac{\sqrt{6}}{6}y\right),$ for $t\in[0,1]$ and $(x,y)\in[0,1]^{2}.$ The initial and boundary conditions are determined by this solution.\\

        To analyze the convergence rate of our numerical scheme, we take the mesh size $h\in\{\frac{1}{2},\frac{1}{2^{2}},\frac{1}{2^{3}},
        \frac{1}{2^{4}},\frac{1}{2^{5}}\}$ and time step $k\in\{\frac{1}{2^{2}},\frac{1}{2^{3}},\frac{1}{2^{4}},\frac{1}{2^{5}},\frac{1}{2^{6}},
        \frac{1}{2^{7}},\frac{1}{2^{8}},\frac{1}{2^{9}},\frac{1}{2^{10}}\frac{1}{2^{11}}\},$ by a mid-point refinement. Under the time step restriction $(\ref{44}),$ we set $k=\frac{1}{2}h^{2}$ and we compute the error estimates: $\||E(u)|\|_{L^{2}(0,T;L^{2})},$ $\||E(u)|\|_{L^{\infty}(0,T;L^{2})}$ and $\||E(u)|\|_{L^{1}(0,T;L^{2})}$ related to the time-split method to see that the algorithm is stable, second order accuracy in time and fourth order convergent in space. In addition, we plot the approximate solution, the exact one and the errors versus $n.$ From this analysis, a three-level explicit time-split MaCormack method is both efficient and effective than a two-level linearized compact ADI approach. In fact, although the two-level linearized compact ADI scheme has the same convergent rate (see \cite{wcld}, Theorem $6.6,$ p. $19$) this method requires too much computer times to achieve the solution. Furthermore, when $h$ varies in the given range, we observe from Tables $1$-$6$ that the approximation errors $O(k^{\beta})+O(h^{\theta})$ are dominated by the h-terms $O(h^{\theta})$ (or $k$-terms $O(k^{\beta})$). So, the ratio $r^{m}_{u},$ where $m=1,2,\infty,$ of the approximation errors on two adjacent mesh levels $\Omega_{2h}$ and $\Omega_{h}$ is approximately $(2h)^{\theta}/h^{\theta}=2^{\theta},$ where $m$ refers to the $L^{m}(0,T;L^{2}(\Omega)$-error norm. Hence, we can simply use $r^{m}_{u}$ to estimate the corresponding convergence rate with respect to $h.$ Define the norms for the approximate solution $u,$ the exact one $\overline{u},$ and the errors $E(u),$ as follows
         \begin{equation*}
         \||u|\|_{L^{2}(0,T;L^{2})}=\left[k\underset{n=0}{\overset{N}\sum}\|u^{n}\|_{L_{f}^{2}}^{2}\right]^{\frac{1}{2}};
         \text{\,\,}\||\overline{u}|\|_{L^{2}(0,T;L^{2})}=\left[k\underset{n=0}{\overset{N}\sum}\|\overline{u}^{n}\|_{L_{f}^{2}}^{2}\right]^{\frac{1}{2}};
         \end{equation*}
         \begin{equation*}
           \||E(u)|\|_{L^{2}(0,T;L^{2})}=\left[k\underset{n=0}{\overset{N}\sum}\|u^{n}-\overline{u}^{n}\|_{L_{f}^{2}}^{2}\right]^{\frac{1}{2}};
           \text{\,\,}\||E(u)|\|_{L^{1}(0,T;L^{2})}=k\underset{n=0}{\overset{N}\sum}\|u^{n}-\overline{u}^{n}\|_{L_{f}^{2}};
         \end{equation*}
         and
         \begin{equation*}
            \||E(u)|\|_{L^{\infty}(0,T;L^{2})}=\underset{0\leq n\leq N}{\max}\|u^{n}-\overline{u}^{n}\|_{L_{f}^{2}}.
         \end{equation*}

          $\bullet$ \textbf{Test $1.$} Let $\Omega$ be the unit square $(0,1)\times(0,1)$ and $T$ be the final time, $T=1.$ We assume
       that the diffusive coefficient $a=1,$ and we choose the force $f(\overline{u})=(1-\overline{u})\overline{u}^{2},$ in such a way that the exact solution $\overline{u}$ is given by
       \begin{equation*}
         \overline{u}(x,y,t)=\left[1+\exp\left(-\frac{1}{2}t+\frac{\sqrt{3}}{3}x+\frac{\sqrt{6}}{6}y\right)\right]^{-1}.
       \end{equation*}
       The initial and boundary conditions are given by this solution. We take the mesh size and time step: $h\in\{\frac{1}{2},\frac{1}{2^{2}},
       \frac{1}{2^{3}},\frac{1}{2^{4}},\frac{1}{2^{5}}\}$ and $k\in\{\frac{1}{2^{2}},\frac{1}{2^{3}},\frac{1}{2^{4}},\frac{1}{2^{5}},
       \frac{1}{2^{6}},\frac{1}{2^{7}},\frac{1}{2^{8}},\frac{1}{2^{9}},\frac{1}{2^{10}},\frac{1}{2^{11}}\}.$\\

           \textbf{Tables 1,2.} Analyzing of convergence rate $O(h^{\theta}+\Delta t^{\beta})$ for time-split MacCormack by $r^{m}_{u},$ with varying time step $k=\Delta t$ and mesh grid $h=\Delta x$. \\

            \textbf{Case: $k=\frac{1}{2}h^{2}$.}
           $$\begin{tabular}{|c|c|c|c|c|c|c|}
            \hline
            $h$ & $ \||E(u)|\|_{L^{2}}$ & $r^{2}_{u}$ & $ \||E(u)|\|_{L^{\infty}}$ & $r^{\infty}_{u}$ & $ \||E(u)|\|_{L^{1}}$ & $r^{1}_{u}$ \\
            \hline
            $2^{-1}$ & 0.0054 & ----  & 0.0058 & ----  & 0.0053 & ---- \\
            \hline
            $2^{-2}$ & 0.0014 & 3.8571 & 0.0014 & 4.1429 & 0.0014 & 3.7857\\
            \hline
            $2^{-3}$ & $0.372\times10^{-3}$ & 3.7634 & $0.3849\times10^{-3}$ & 3.6373 & $0.3717\times10^{-3}$ & 3.7665\\
            \hline
            $2^{-4}$ & $0.966\times10^{-4}$ & 3.8509 & $0.995\times10^{-4}$ & 3.8683 & $0.963\times10^{-4}$ & 3.8598\\
            \hline
            $2^{-5}$ & $0.2459\times10^{-4}$ & 3.9284 & $0.2529\times10^{-4}$ & 4.0963 & $0.2450\times10^{-4}$ & 3.9306\\
            \hline
          \end{tabular}$$
           \text{\,}\\
              \textbf{ Case: $k=h^{2}$.}
          $$\begin{tabular}{|c|c|c|c|c|c|c|}
            \hline
            $h$ & $ \||E(u)|\|_{L^{2}}$ & $r^{2}_{u}$ & $ \||E(u)|\|_{L^{\infty}}$ & $r^{\infty}_{u}$ & $ \||E(u)|\|_{L^{1}}$ & $r^{1}_{u}$ \\
            \hline
            $2^{-1}$ & 0.0200 &  ---- & 0.0227 & ----  & 0.0189 &  ---- \\
            \hline
            $2^{-2}$ & 0.0050 & 4.0000 & 0.0069 & 3.2899 & 0.0049 & 3.8571 \\
            \hline
            $2^{-3}$ & NAN    & ---- & inf    & ---- & Nan    & ----   \\
            \hline
          \end{tabular}$$

        $\bullet$ \textbf{Test $2.$} Now, let $\Omega$ be the unit square $(0,1)^{2}$ and $T=1.$ The diffusive term $a$ is assumed equals $1.$ We
       choose the force $f$ such that the analytic solution $\overline{u}$ is defined as
       \begin{equation*}
         \overline{u}(x,y,t)=1+\exp\left(-\frac{1}{2}t+\frac{\sqrt{3}}{3}x+\frac{\sqrt{6}}{6}y\right),\text{\,\,\,\,\,and\,\,\,\,\,\,}f(\overline{u})=
         1-\overline{u}.
       \end{equation*}
       The initial and boundary conditions also are given by the exact solution $\overline{u}.$ Similar to \textbf{Test $1,$} we take the mesh size and time step: $h\in\{\frac{1}{2},\frac{1}{2^{2}},\frac{1}{2^{3}},\frac{1}{2^{4}},\frac{1}{2^{5}}\}$ and $k\in\{\frac{1}{2^{2}},\frac{1}{2^{3}},
       \frac{1}{2^{4}},\frac{1}{2^{5}},\frac{1}{2^{6}},\frac{1}{2^{7}},\frac{1}{2^{8}},\frac{1}{2^{9}},\frac{1}{2^{10}},\frac{1}{2^{11}}\}.$\\

          \textbf{Tables 3,4.} Convergence rates $O(h^{\theta}+\Delta t^{\beta})$ for time-split MacCormack by $r^{m}_{u},$ with varying
          spacing $h$ and time step $k$.\\

           \textbf{Case: $k=\frac{1}{2}h^{2}$.}

           $$\begin{tabular}{|c|c|c|c|c|c|c|}
            \hline
            $h$ & $ \||E(u)|\|_{L^{2}}$ & $r^{2}_{u}$ & $ \||E(u)|\|_{L^{\infty}}$ & $r^{\infty}_{u}$ & $ \||E(u)|\|_{L^{1}}$ & $r^{1}_{u}$ \\
            \hline
            $2^{-1}$ & 0.0245 &        & 0.0310 &        & 0.0232 &  \\
            \hline
            $2^{-2}$ & 0.0060 & 4.0833 & 0.0072 & 4.3056 & 0.0060 & 3.8667\\
            \hline
            $2^{-3}$ & $0.16\times10^{-2}$ & 3.75 & $0.19\times10^{-2}$ & 3.7895 & $0.16\times10^{-2}$ & 3.75\\
            \hline
            $2^{-4}$ & $0.4\times10^{-3}$ & 4.0000 & $0.5\times10^{-3}$ & 3.8000 & $0.4\times10^{-3}$ & 4.0000\\
            \hline
            $2^{-5}$ & $0.1061\times10^{-3}$ & 3.7700 & $0.1248\times10^{-3}$ & 4.0064 & $0.1052\times10^{-3}$ & 3.8023\\
            \hline
          \end{tabular}$$
           \text{\,}\\
           \textbf{Case: $k=h^{2}$.}
           $$\begin{tabular}{|c|c|c|c|c|c|c|}
            \hline
            $h$ & $ \||E(u)|\|_{L^{2}}$ & $r^{2}_{u}$ & $ \||E(u)|\|_{L^{\infty}}$ & $r^{\infty}_{u}$ & $ \||E(u)|\|_{L^{1}}$ & $r^{1}_{u}$ \\
            \hline
            $2^{-1}$ & 0.0892 &        & 0.1163 &        & 0.0800 &  \\
            \hline
            $2^{-2}$ & 0.0364 & 2.4505 & 0.0814 & 1.4287 & 0.0320 & 2.5000\\
            \hline
            $2^{-3}$ & $0.2132\times10^{20}$ & ---- & $1.5596\times10^{20}$ & ---- & $0.0409\times10^{20}$ & ---- \\
            \hline
          \end{tabular}$$
           \text{\,}\\
           \text{\,}\\
       $\bullet$ \textbf{Test $3.$} Finally, let $\Omega$ be the unit square $(0,1)\times(0,1)$ and $T=1.$ We assume that $a=1,$ and
       the force $f$ is chosen such that the exact solution $u$ is given by
       \begin{equation*}
         \overline{u}(x,y,t)=\frac{1}{2}+\frac{1}{2}\tanh\left(\frac{3}{4}t+\frac{1}{4}x+\frac{1}{4}y\right),\text{\,\,\,\,\,and\,\,\,\,\,\,}
         f(\overline{u})=(1-\overline{u}^{2})\overline{u}.
       \end{equation*}
       The initial and boundary conditions are given by the exact solution $\overline{u}$.\\

       Similar to both \textbf{Tests $1,2$} the mesh size and time step are chosen such that: $h\in\{\frac{1}{2},\frac{1}{2^{2}},\frac{1}{2^{3}},
       \frac{1}{2^{4}},\frac{1}{2^{5}}\}$ and $k\in\{\frac{1}{2^{2}},\frac{1}{2^{3}},\frac{1}{2^{4}},\frac{1}{2^{5}},\frac{1}{2^{6}},\frac{1}{2^{7}},
       \frac{1}{2^{8}},\frac{1}{2^{9}},\frac{1}{2^{10}},\frac{1}{2^{11}}\},$ by a mid-point refinement. We compute the error estimates: $E(u)$ related to a three-level explicit time-split MacCormack approach to see that the algorithm is second order convergent in time and fourth order accurate in space.
       Furthermore, we plot the errors together with the energies versus $n.$ From this analysis, it is obvious that a three-level time-split scheme is efficient and effective than a two-level linearized compact ADI method which has the same order of convergence.\\

          \textbf{Tables 5,6.} Convergence rates $O(h^{\theta}+\Delta t^{\beta})$ for time-split MacCormack by $r^{m}_{u},$ with varying spacing $h$ and
          time step $\Delta t$.\\

           \textbf{Case: $k=\frac{1}{2}h^{2}$.}
           $$\begin{tabular}{|c|c|c|c|c|c|c|}
            \hline
            $h$ & $ \||E(u)|\|_{L^{2}}$ & $r^{2}_{u}$ & $ \||E(u)|\|_{L^{\infty}}$ & $r^{\infty}_{u}$ & $ \||E(u)|\|_{L^{1}}$ & $r^{1}_{u}$ \\
            \hline
            $2^{-1}$ & 0.0112 &        & 0.0151 &        & 0.0106 &  \\
            \hline
            $2^{-2}$ & 0.0029 & 3.8621 & 0.0036 & 4.1944 & 0.0028 & 3.7857\\
            \hline
            $2^{-3}$ & $0.8\times10^{-3}$ & 3.6250 & $0.1\times10^{-2}$ & 3.6000 & $0.8\times10^{-3}$ & 3.5000\\
            \hline
            $2^{-4}$ & $0.2001\times10^{-3}$ & 3.9980 & $0.2506\times10^{-3}$ & 3.9904 & $0.5496\times10^{-3}$ & 4.0796\\
            \hline
            $2^{-5}$ & $0.509\times10^{-4}$ & 3.9312 & $0.638\times10^{-4}$ & 3.9279 & $0.5000\times10^{-4}$ & 3.9220\\
            \hline
          \end{tabular}$$
           \text{\,}\\
           \textbf{Case: $k=h^{2}$.}
           $$\begin{tabular}{|c|c|c|c|c|c|c|}
            \hline
            $h$ & $ \||E(u)|\|_{L^{2}}$ & $r^{2}_{u}$ & $ \||E(u)|\|_{L^{\infty}}$ & $r^{\infty}_{u}$ & $ \||E(u)|\|_{L^{1}}$ & $r^{1}_{u}$ \\
            \hline
            $2^{-1}$ & 0.0400 &        & 0.0546 &        & 0.0363 &  \\
            \hline
            $2^{-2}$ & 0.0150 & 2.6667 & 0.6601 & 1.6957 & 0.5821 & 2.6691\\
            \hline
            $2^{-3}$ & NaN & ---- & Inf & ---- & NaN & ---- \\
            \hline
          \end{tabular}$$

            The analysis on the convergence of the numerical scheme presented in Section $\ref{sec4},$ has suggested that our algorithm is first order convergent in time and fourth order accurate in space. If the result provided in Section $\ref{sec2},$ page $6$, last paragraph is to believe, this shows that the time-split MacCormack scheme is inconsistent. Surprisingly, it comes from \textbf{Tests 1-3}, more precisely Figures $\ref{fig1}$-$\ref{fig3}$ and \textbf{Tables 1-6}, that the three-level explicit time-split MacCormack technique is stable, second order accurate in time and fourth order convergent in space under the time step restriction $(\ref{44}),$ which confirms the theoretical result provided in Section $\ref{sec2},$ page $6$, last paragraph. Thus, the considered method applied to initial-boundary value problem $(\ref{1})$-$(\ref{3})$ is: stable, consistent, second order convergent in time and fourth order accurate in space.

         \section{General conclusion and future works}\label{sec6}
         We have studied in detail the stability, error estimates and convergence rate of a three-level explicit time-split MacCormack method for solving the $2$D nonlinear reaction-diffusion equation $(\ref{1})$-$(\ref{3})$. The analysis has suggested that our method is stable, consistent, second order accuracy in time and fourth order convergent in space under the time step restriction $(\ref{44})$. This convergence rate is confirmed by a large set of numerical experiments (see both Figures $\ref{fig1}$-$\ref{fig3}$ and \textbf{Tables 1-6}). Numerical evidences also show that the new algorithm is: (1) more efficient and effective than a two-level linearized compact ADI method, (2) fast and robust tools for the integration of general systems of parabolic PDEs. However, the time-split MacCormack method is not is a satisfactory approach for solving high Reynolds number flows where the viscous region becomes very thin. For these flows, the mesh grid must be highly refined in order to accurately resolve the viscous regions. This leads to very small time steps and subsequently long computing times. To overcome this difficulty, MacCormack developed a hybrid version of his scheme, which is known as MacCormack rapid solver method \cite{mc1}. This hybrid scheme is an explicit-implicit method which has been proved to be from $10$ to $100$ more faster than a time-split MacCormack algorithm (see \cite{apt}, P. 632). The rapid solver method will be applied to the two-dimensional nonlinear reaction-diffusion equations in our future works.

          \begin{figure}
         \begin{center}
          Analysis of stability and convergence of a three-level explicit time-split MacCormack method with $a=1$.
          \psfig{file=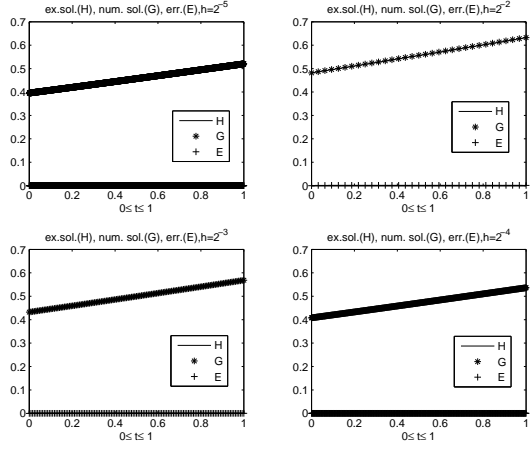,width=7cm}
        \end{center}
         \caption{$\overline{u}(x,y,t)=\left[1+\exp\left(-\frac{1}{2}t+\frac{\sqrt{3}}{3}x+\frac{\sqrt{6}}{6}y\right)\right]^{-1}$ and $f(\overline{u})=(1-\overline{u})\overline{u}^{2}$}
          \label{fig1}
          \end{figure}

           \begin{figure}
         \begin{center}
          Analysis of stability and convergence of a three-level explicit time-split MacCormack method with $a=1$.
           \psfig{file=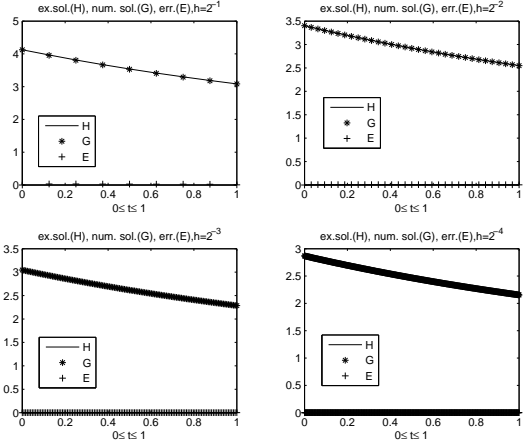,width=7cm}
        \end{center}
         \caption{ $\overline{u}(x,y,t)=1+\exp\left(-\frac{1}{2}t+\frac{\sqrt{3}}{3}x+\frac{\sqrt{6}}{6}y\right)$ and $f(\overline{u})=1-\overline{u}$}
          \label{fig2}
          \end{figure}

           \begin{figure}
         \begin{center}
          Analysis of stability and convergence of a three-level explicit time-split MacCormack method with $a=1$.
          \psfig{file=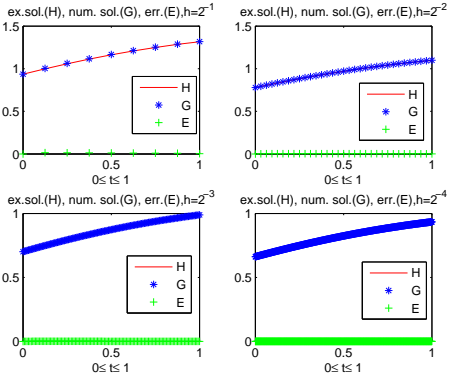,width=7cm}
        \end{center}
         \caption{$\overline{u}(x,y,t)=\frac{1}{2}+\frac{1}{2}\tanh\left(\frac{3}{4}t+\frac{1}{4}x+\frac{1}{4}y\right)$ and $f(\overline{u})=(1-\overline{u}^{2})\overline{u}$}
          \label{fig3}
          \end{figure}

     \end{document}